\documentclass{gtart}

\usepackage[inner=30mm, outer=30mm, textheight=245mm]{geometry}
\usepackage{psfrag, graphicx, epsfig, color}
\usepackage{amsmath, amssymb, latexsym, euscript, amsthm}
\usepackage{mathptmx, wrapfig}
\usepackage{mathrsfs}
\usepackage{booktabs}
\usepackage[small,bf]{caption}
\usepackage{multicol}
\usepackage{array}
\usepackage{caption}
\usepackage{subcaption}

\usepackage[colorlinks, citecolor=blue, linkcolor=blue, backref=page]{hyperref}
\usepackage[nameinlink]{cleveref}

\usepackage[mode=buildnew]{standalone}

\usepackage{tikz}
\usetikzlibrary{positioning}

\usepackage[shortlabels]{enumitem}

\usepackage{xcolor}

\usepackage{ragged2e}

\usepackage{siunitx}

\usepackage[T1]{fontenc}

\theoremstyle{plain}
\newtheorem{theorem}{Theorem}
\newtheorem*{theorem*}{Theorem}
\newtheorem{lemma}[theorem]{Lemma}
\newtheorem{proposition}[theorem]{Proposition}
\newtheorem{corollary}[theorem]{Corollary}

\newtheorem*{claim*}{Claim}

\theoremstyle{definition}
\newtheorem{definition}[theorem]{Definition}
\newtheorem*{definition*}{Definition}
\newtheorem{algorithm}[theorem]{Algorithm}

\theoremstyle{remark}
\newtheorem{remark}[theorem]{Remark}
\newtheorem{example}{Example}

\numberwithin{equation}{section}

\begin{document}

\title{\Large An algorithm to construct one-vertex triangulations of\\
Heegaard splittings}
\author{Alexander He, James Morgan, Em K. Thompson}

\begin{abstract}
Following work of Jaco and Rubinstein (2006), which (non-constructively) proved that any 3-manifold admits a one-vertex layered triangulation,
we present an algorithm, with implementation using Regina, that uses a combinatorial presentation
of a Heegaard diagram to construct a generalised notion of a layered triangulation.
We show that work of Husz\'ar and Spreer (2019) extends to our construction:
given a genus-$g$ Heegaard splitting, our algorithm generates a triangulation with cutwidth bounded above by $4g-2$.
Beyond Heegaard splittings, our construction actually extends to a natural generalisation of Dehn fillings:
given a one-vertex triangulation with a genus-$g$ boundary component $B$, we can construct
a one-vertex triangulation of any 3-manifold obtained by filling $B$ with a handlebody.
To demonstrate the usefulness of our algorithm, we present findings from preliminary computer searches using this algorithm.
\end{abstract}

\makeshorttitle

\section{Introduction}\label{sec:intro}

Layered triangulations of 3-manifolds were introduced by Jaco and Rubinstein in 2006~\cite{JacoRubinstein2006}. The primary focus of their paper was on the construction and classification of \textit{layered solid tori}, including a classification of all normal and almost normal surfaces contained within them.
They also constructed layered handlebodies with genus $g>1$, but did not investigate these to the same extent
because the construction suffers from combinatorial explosion as genus increases.

A layered handlebody of any genus contains a single vertex, by definition (see Section~\ref{sec:JRdef}),
and can be adjusted to have any desired one-vertex boundary triangulation up to isotopy.
Meanwhile, all closed, orientable 3-manifolds admit a Heegaard splitting for some genus~\cite{moise_affine_1952}.
Using these two nontrivial facts, Jaco and Rubinstein show~\cite[pp.~91--93]{JacoRubinstein2006} that every closed 3-manifold $M$ admits
a (one-vertex) layered triangulation: that is, a triangulation of $M$ built by gluing two layered handlebodies.

A useful feature of layered triangulations is that the \emph{cutwidth} is bounded above by ${4g-2}$, where $g$ is the genus of the Heegaard splitting~\cite[Theorem~1]{Huszar-Spreer};
we will elaborate on cutwidth shortly, in Section~\ref{subsec:introComplexity}.
Unfortunately, this feature is difficult to exploit in practice because Jaco and Rubinstein do not give an explicit construction for turning a Heegaard splitting into a layered triangulation.

Our contribution is a generalised notion of layered triangulation that exhibits the same bound on cutwidth.
Unlike Jaco and Rubinstein's construction, which starts on the inside of a handlebody and builds outwards,
our construction starts on the boundary of a handlebody and builds inwards.
This allows us to give a natural topological description of our construction:
given a closed orientable surface $S$ marked with attaching circles, we fill $S$ with a handlebody $H$ so that each of the attaching circles bounds a disc in $H$.
Moreover, our construction has the advantage that it proceeds via an explicit algorithm that is implementable in software.

In principle, our algorithm is easily adapted to construct one-vertex triangulations in a variety of settings:
\begin{enumerate}[(A)]
\item\label{settingA}
As outlined above, we can fill a surface with a handlebody, and hence build a triangulation of a handlebody with a prescribed boundary triangulation.
\item\label{settingB}
Given a Heegaard splitting specified by a Heegaard diagram---that is, a closed orientable surface $S$ marked with two sets of attaching circles,
where each set specifies how to attach a handlebody on one of the sides of $S$---we can attach handlebodies to both sides of the splitting surface.
The result is a triangulation of the closed $3$-manifold corresponding to the given splitting.
\item\label{settingC}
Given any compact orientable $3$-manifold $M$ with boundary, together with a set of attaching circles in a boundary component $B$ of $M$,
we can build a triangulation of the $3$-manifold obtained by filling $B$ in accordance with the attaching circles.
The result is a natural generalisation of a Dehn filling;
such generalised fillings have been studied previously, for example by Lackenby~\cite{Lackenby_2002} and Easson~\cite{Easson2006}.
\end{enumerate}

In terms of software implementation, building a $3$-manifold triangulation in Setting~\ref{settingC} is
the most straightforward because the starting point is already $3$-dimensional.
This setting also gives a different way to build a triangulation from a Heegaard splitting:
rather than attaching two handlebodies to a surface, we can simply start with a handlebody and fill in the handlebody on the other side.
With this in mind, we present an algorithm for filling a genus-$g$ boundary component of
a one-vertex 3-manifold triangulation, which can be roughly formulated as follows:
\begin{algorithm}\label{algm:intro}
Input: a one-vertex triangulation of a compact orientable 3-manifold with a genus-$g$ boundary component, together with a set of attaching circles on the boundary.
\begin{enumerate}
    \item Layer tetrahedra onto the genus-$g$ boundary to adjust the boundary triangulation
    until all the attaching circles appears as diagonals of disjoint quadrilaterals.
    \item Perform folds to identify pairs of boundary faces in such a way that each attaching circle is made homotopically trivial.
    \item Using layering and folding, close the triangulation without introducing new topology.
    \item Terminate and return the resulting triangulation.
\end{enumerate}
\end{algorithm}

The algorithm is described more carefully in Section~\ref{sec:main-algorithm}.
It is implemented using Regina~\cite{Regina}, and is available at \url{https://github.com/AlexHe98/heegaardbuilder/};
in future, we hope to incorporate a version of our algorithm as part of Regina.
We give a preliminary demonstration of the usefulness of this algorithm by performing some computational experiments in Section~\ref{sec:experimentation}.

Currently, our implementation accepts any one-vertex triangulation of a compact orientable $3$-manifold as input;
the one-vertex restriction means that, at least for now, our implementation only works for $3$-manifolds with exactly one boundary component.
The implementation also does not cover Settings~\ref{settingA} and~\ref{settingB}, where we start with a surface instead of a $3$-manifold.
In future, we hope to extend our implementation to cover these other cases;
this would be tedious, but should otherwise be conceptually straightforward.

To our knowledge, the only preexisting software that is comparable to our algorithm is Twister~\cite{Twister},
which has been incorporated in SnapPy~\cite{SnapPy} since version 1.3.10.
One of the key features of Twister is that it converts Heegaard splittings into triangulations, but the input is described very differently from our algorithm:
for Twister, the input is a mapping class of the Heegaard surface, specified using a composition of Dehn twists and half twists.
Beyond the difference in input, our construction is also distinguished from Twister by some appealing properties that we have already mentioned:
\begin{enumerate}[(a)]
\item
Our construction generalises the layered construction of Jaco and Rubinstein~\cite{JacoRubinstein2006};
in particular, in the genus-$1$ case, our construction coincides exactly with the layered construction.
\item
The output of our algorithm is always a one-vertex triangulation whose cutwidth and treewidth are bounded above by $4g-2$, where $g$ is the genus of the input Heegaard splitting.
\end{enumerate}

There is also some related work that has not been implemented in software.
In Theorem~27 of~\cite{HuszarSpreer2018WidthFull} (the full version of~\cite{Huszar-Spreer}),
Husz\'ar and Spreer used work from Bell's thesis~\cite[Section~2.4]{Bell2015Thesis}
to give an algorithm that converts a Heegaard splitting into a layered triangulation;
thus, the output has essentially the same appealing properties as the output of our algorithm.
However, their input is similar to that used in Twister:
the Heegaard splitting is described using a mapping class, which is written as a word encoding a composition of Dehn twists.

\subsection{Measuring complexity of triangulations}\label{subsec:introComplexity}

Triangulations are an important tool for representing $3$-manifolds, particularly for the purpose of computation. In practice, one-vertex triangulations play an especially important role. One reason for this is that one-vertex triangulations typically require significantly fewer tetrahedra than less flexible triangulations like simplicial complexes; this is crucial because many $3$-manifold algorithms have running times that scale exponentially with the number of tetrahedra in the input.

In the past decade or so, there has been increasing interest in measuring the complexity of triangulations not by number of vertices or number of tetrahedra, but rather by a quantity known as \emph{treewidth}. Loosely speaking, the lower the treewidth of a triangulation, the more ``tree-like'' the dual graph of the triangulation. The interest in treewidth is driven by algorithms that are \emph{fixed-parameter tractable} in the treewidth,
meaning that the computation time is bounded above by a product of a polynomial function of the number of tetrahedra and an arbitrary function of the treewidth; this means that for any fixed treewidth, the computation time reduces to a polynomial function, and hence the algorithm becomes ``tractable''.
For examples of such fixed-parameter tractable algorithms for solving various problems in $3$-manifold topology, see~\cite{BurtonDowney2017,BLPS2016FPTMorse,BMS2018,BurtonPettersson2014FPT,BurtonSpreer2013}.

One of the main motivations for constructing layered triangulations is because of their relatively low width.
As already alluded to, Husz\'ar and Spreer~\cite[Theorem~1]{Huszar-Spreer} showed that layered triangulations have bounded \textit{cutwidth} (in terms of the genus of the Heegaard splitting), which is itself an upper bound for treewidth.
We define cutwidth and discuss complexity in more detail in Section~\ref{sec:complexity}, but for now let us state the following.

\begin{theorem}\label{thm:cutwidth-intro}
    Given a genus-$g$ Heegaard splitting of a closed, orientable $3$-manifold $M$, Algorithm~\ref{algm:intro} generates a triangulation of $M$ that is one-vertex and has cutwidth bounded above by ${4g-2}$.
\end{theorem}

For a \textit{non-Haken} $3$-manifold $M$ with Heegaard genus equal to $g$, Huszar, Spreer and Wagner showed that any triangulation of $M$ must
have treewidth at least $\frac{1}{18}g-1$~\cite[Theorem~4]{HSW2019JoCGTreewidth}.
Combining this with Theorem~\ref{thm:cutwidth-intro}, we see that if we have a minimum genus Heegaard splitting of a non-Haken $3$-manifold $M$,
then we can use Algorithm~\ref{algm:intro} to convert this into a triangulation of $M$ whose treewidth is optimal up to a constant factor.

\subsection{Outline of paper}
Section~\ref{sec:background} covers relevant background on layered triangulations.
In Section~\ref{sec:HeegaardDiagrams}, we: \emph{(a)} formalise \textit{filling diagrams} and interpret Heegaard diagrams in this context;
\emph{(b)} describe a pinched version of a filling diagram;
and \emph{(c)} introduce the \textit{combinatorial filling diagram}, to be used as input for our main algorithm.
In Section~\ref{sec:main-algorithm}, we describe each component of the main algorithm in detail,
ultimately proving that it terminates and returns a valid triangulation of the desired 3-manifold.
We return to the discussion of triangulation complexity in Section~\ref{sec:complexity}, and prove Theorem~\ref{thm:cutwidth-intro}.
In Section~\ref{sec:implementation}, we: \emph{(a)} give instructions for using the algorithm in Regina;
\emph{(b)} comment on certain aspects of the implementation; and \emph{(c)} work through three examples of genus-$2$ Heegaard splittings.
We conclude in Section~\ref{sec:experimentation} with a summary of findings from preliminary computer searches using our algorithm.

\subsection*{Acknowledgements}
    We are extremely grateful to the MATRIX Institute and everyone involved in the
    \textit{Graduate Talks in Geometry and Topology Get-Together}, where this collaboration began.
    We acknowledge the input of Layne Hall, Connie Hui and Lecheng Su, who were involved in the project during our week at MATRIX in 2022.
    We would each like to thank our PhD supervisors Ben Burton, Jonathan Spreer,
    Stephan Tillmann and Jessica Purcell for helpful conversations and feedback along the way.
    We also thank the anonymous reviewers for their detailed feedback, which has significantly improved the readability and accessibility of this paper.
    Each author was supported by an Australian Government Research Training Program (RTP) Scholarship.

\section{Background}\label{sec:background}
In this section we recall relevant definitions and fix notation to be used throughout the paper.

\subsection{Triangulations}

A \textit{$3$-dimensional triangulation}, or \textit{$3$-triangulation} for short, is a finite collection of tetrahedra,
together with instructions for how to \textit{glue} (that is, affinely identify) some or all of their triangular faces together in pairs.
The faces that are left unglued form the \emph{boundary} of the triangulation.
We allow the tetrahedra to be flexible, since the topology of the triangulation does not depend on the shape of the tetrahedra;
in particular, we allow two faces of the same tetrahedron to be glued to each other.
Such triangulations are sometimes referred to as \textit{generalised triangulations}, \textit{semisimplicial triangulations} or \textit{pseudotriangulations}, to emphasise the fact that they are less strict than simplicial complexes.

Similarly, a \textit{$2$-dimensional triangulation}, or \textit{$2$-triangulation} for short, is a finite collection of triangles,
together with instructions for how to \textit{glue} some or all of their \emph{edges} together in pairs.
Again, we allow two edges of the same triangle to be glued to each other.
When the dimension is either unimportant or clear from context, we may use the word
\textit{triangulation} to refer to either a $3$-triangulation or a $2$-triangulation.

One of the key features of this flexible notion of triangulation is that it is possible to build a \emph{one-vertex} triangulation:
a triangulation in which, as a consequence of the face-gluings, all of the vertices are identified to become one.
Throughout this paper, all our triangulations will be one-vertex, unless stated otherwise.

In a $3$-triangulation $\mathcal{T}$, the \textit{link} of a vertex $v$ is the surface given by the frontier of a small regular neighbourhood of $v$ in $\mathcal{T}$.
The vertex $v$ is \textit{valid} if its link is either a disc (if $v$ lies in the boundary) or a closed surface (if $v$ is internal to the triangulation);
otherwise, $v$ is \textit{invalid}.

An edge is \emph{invalid} if it is identified to itself in reverse, and \emph{valid} otherwise.
In general, a 3-triangulation is \emph{valid} if all of its vertices and edges are valid, and \emph{invalid} otherwise.
However, the triangulations in this paper will never have invalid edges, so for our purposes validity of triangulations will depend entirely on the vertex links.

For a valid (one-vertex) triangulation, the underlying topological space is a 3-manifold if and only if
the vertex is \emph{real} (or \emph{material}), meaning that the link is either a disc or a 2-sphere.
In general, it is possible for a vertex to be valid but not real, though we will not encounter such cases in this paper.

However, we will frequently encounter triangulations that have invalid vertices
(specifically, such triangulations arise as intermediate triangulations constructed by our algorithm, even though the input and output triangulations are valid).
The underlying topological space of such a triangulation is not quite a $3$-manifold, so we introduce some terminology for such a space.
We define a \textit{$3$-manifold with invalid points} to be a topological space $M$ with a finite set $I$ of marked points such that $M-I$ is a $3$-manifold,
but the link of each point in $I$ is neither a closed surface nor a disc;
we call each point in $I$ an \textit{invalid point}.

There are two important operations on triangulations that are central to this paper: layering and folding.

\begin{definition}[Layering]
Let \textbf{\textit{e}} be an edge in a $2$-triangulation $B$.
Roughly, we \textit{layer over} \textbf{\textit{e}} by attaching a tetrahedron $\Delta$ to the two faces of $B$ adjacent to \textbf{\textit{e}}.

To make this precise, fix an orientation of \textbf{\textit{e}}, and let $f_1$ and $f_2$ denote the (not necessarily distinct) triangles of $B$ adjacent to \textbf{\textit{e}}.
Likewise, choose an edge $\textbf{\textit{e}}^\ast$ of $\Delta$, fix an orientation of $\textbf{\textit{e}}^\ast$, and let $f_1^\ast$ and $f_2^\ast$ denote the triangles of $\partial\Delta$ adjacent to $\textbf{\textit{e}}^\ast$.
For our purposes, we only need to define layering in the following cases:
\begin{itemize}
\item
$f_1 \neq f_2$ (see Figure~\ref{fig:layering}); or
\item
$f := f_1 = f_2$, where \textbf{\textit{e}} is given by identifying two edges of the triangle $f$ with a twist,
so that \textbf{\textit{e}} forms the internal edge of a one-triangle M\"obius band (see Figure~\ref{fig:layerOneTriangleMobius}).
\end{itemize}
The layering is given by affinely identifying $f_i$ to $f_i^\ast$, for each $i\in\{1,2\}$,
in such a way that \textbf{\textit{e}} and $\textbf{\textit{e}}^\ast$ are identified with matching orientations.
In the case where $f_1 = f_2$, this effectively causes $f_1^\ast$ and $f_2^\ast$ to be identified to each other with a twist, so that $\Delta$ becomes a one-tetrahedron solid torus.
\end{definition}

\begin{figure}[htbp]
    \centering
\begin{tikzpicture}[inner sep=0pt]
\node at (0,0) {
	\includegraphics[width=0.7\textwidth]{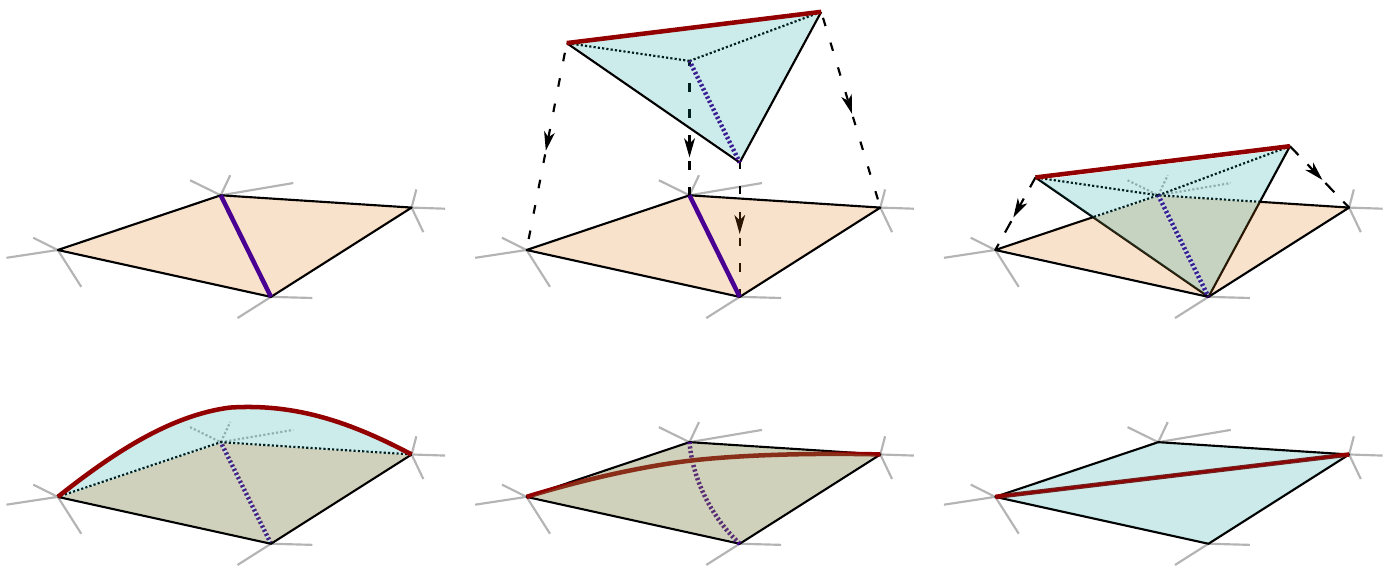}
};
\node at (-3.9,0.37) {$f_1$};
\node at (-3.15,0.375) {$f_2$};
\node at (0.05,2.25) {$\textbf{\textit{e}}'$};
\node at (0,0.345) {$\textbf{\textit{e}}$};
\node at (0.35,1.525) {$\textbf{\textit{e}}^\ast$};
\node at (3.865,-1.65) {$\textbf{\textit{e}}'$};
\end{tikzpicture}
    \caption{\textit{Layering over an edge.}
Consider an edge $\textbf{\textit{e}}$ with adjacent faces $f_1$ and $f_2$ in a $2$-triangulation $B$.
To layer over $\textbf{\textit{e}}$ means to attach a tetrahedron to $f_1$ and $f_2$ so that $\textbf{\textit{e}}$ is covered.
If $B$ is the boundary of a $3$-triangulation, then such a layering changes the boundary $2$-triangulation by replacing $\textbf{\textit{e}}$ with the edge $\textbf{\textit{e}}'$;
this change is called a \textit{flip} in the boundary $2$-triangulation.
}
    \label{fig:layering}
\end{figure}

\begin{figure}[htbp]
\centering
	\includegraphics[scale=0.75]{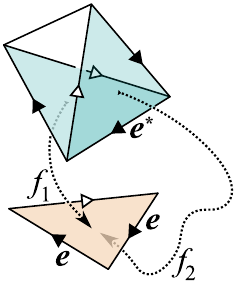}
\caption{\textit{One-tetrahedron solid torus.} Layering a tetrahedron $\Delta$ over the internal edge $\textbf{\textit{e}}$ of a one-triangle M\"obius band. One shaded face of $\Delta$ is attached to the `top' of the M\"obius triangle ($f_1$), then the second shaded face of $\Delta$ is folded around and attached to the `bottom' ($f_2$) \emph{with a twist}. All edges with matching arrows are identified together.}
\label{fig:layerOneTriangleMobius}
\end{figure}

In most applications of layering in this paper, the $2$-triangulation $B$ will be a boundary component of some $3$-triangulation $\mathcal{T}$.
In such a case, observe that layering does not change the topology of $\mathcal{T}$, but it does
adjust the boundary triangulation $B$ by \textit{flipping} the edge \textbf{\textit{e}}.

\begin{definition}[Folding]\label{def:folding}
Let $B$ be a boundary component of a $3$-triangulation $\mathcal{T}$.
Consider a quadrilateral formed by two distinct boundary faces that share a common edge \textbf{\textit{e}} in $B$.
We call \textbf{\textit{e}} the \emph{diagonal} of this quadrilateral;
we call the other diagonal, which is not an edge of $B$, the \emph{off-diagonal} of the quadrilateral.
Note that it is possible, but not necessary, for the off-diagonal to be isotopic to an edge in the interior of $\mathcal{T}$.
We may \emph{fold across} the diagonal \textbf{\textit{e}} to identify the two faces that form the quadrilateral.%
\footnote{When the endpoints of the off-diagonal are given by distinct vertices of the $3$-triangulation (which never happens for our purposes,
since our $3$-triangulations are one-vertex), folding is equivalent to a \emph{book closing move}.}
An example of this is shown in Figure~\ref{fig:folding}.
At times we might instead use the terminology of \textit{folding along $\gamma$},
where $\gamma$ is the curve given by the off-diagonal of the quadrilateral.
Observe that such a fold causes $\gamma$ to become homotopically trivial as a loop.
\end{definition}

\begin{figure}[htbp]
    \centering
\begin{tikzpicture}[inner sep=0pt]
\node at (0,0) {
	\includegraphics[width=0.8\textwidth]{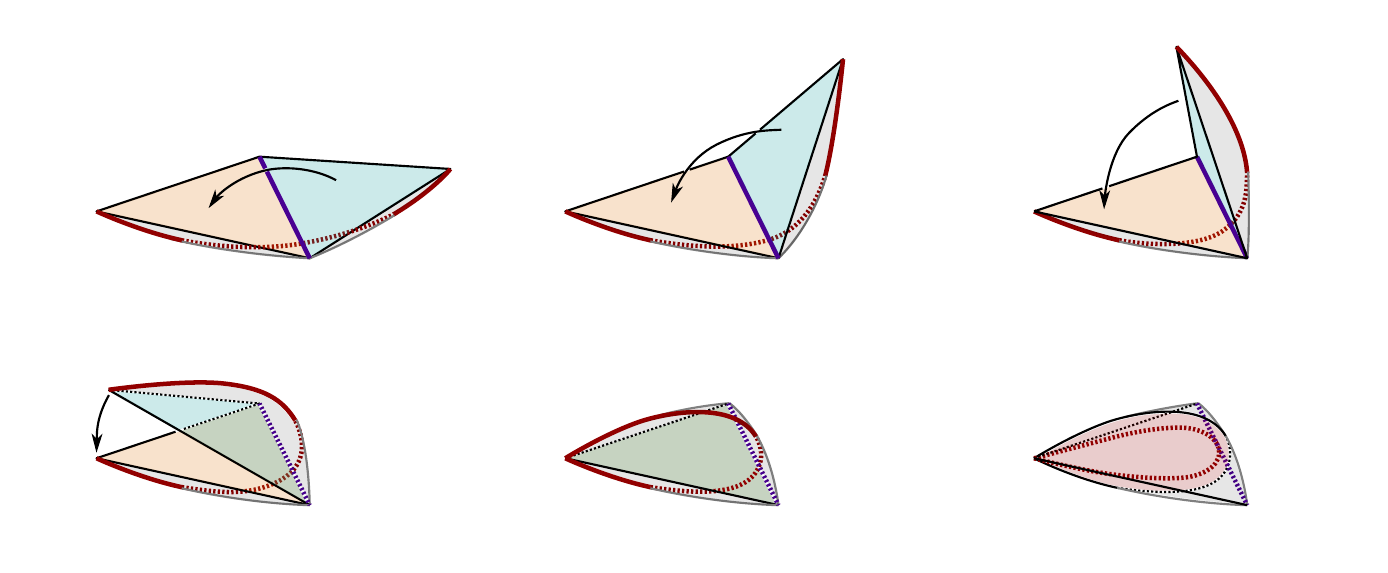}
};
\node at (-3.55,0.8) {$\textbf{\textit{e}}$};
\node at (1.325,0.95) {$\gamma$};
\end{tikzpicture}
    \caption{\textit{Folding across an edge.}
    Folding across the edge $\textbf{\textit{e}}$ identifies the two boundary faces incident to $\textbf{\textit{e}}$.
    The curve $\gamma$ in the interior of the manifold is folded back on itself and becomes homotopically trivial;
    note that $\gamma$ is not necessarily an edge of the triangulation.
    This action may also be described as folding \textit{along} $\gamma$.}
    \label{fig:folding}
\end{figure}

Folding reduces the number of boundary faces by two and, in contrast to layering, \textit{can} change the topology of a triangulation.
Indeed, a single fold can not only change the underlying 3-manifold, but it can also often cause the vertex to become invalid.
Moreover, further folding operations could turn such an invalid vertex back into a valid one.
Such scenarios are not merely hypothetical:
our algorithm necessarily creates intermediate triangulations with invalid vertices, and must subsequently make such vertices valid again.

\subsection{Layered handlebodies
}\label{sec:JRdef}
Let us outline Jaco and Rubinstein's definition of a layered handlebody of genus $g\geqslant 1$;
we will usually refer to this as the \textit{JR-construction}.
The JR-construction was first introduced for arbitrary $g\geqslant 1$ in~\cite[pp.~82--85]{JacoRubinstein2006},
though the $g=1$ case had already appeared in earlier work such as~\cite{JacoRubinstein2003};
also see~\cite[Section~1.2]{Burton2003Thesis} and the references therein for the $g=1$ case.

To start, we define a \textit{$g$-spine} to be any one-vertex triangulation
of a compact surface with one boundary component and Euler characteristic $1-g$.
For example, for $g=1$, the unique $g$-spine is the one-triangle M\"obius band, which has one vertex and two edges:
one interior edge and one boundary edge (see Figure~\ref{fig:g-spines}, left).
In general, note that for odd $g$ the $g$-spine must be a non-orientable surface, but for even $g$ it could be either orientable or non-orientable.
The $g$-spine of a layered handlebody is labelled $\mathcal{T}_{-1}$ in the JR-construction.

\begin{figure}[htbp]
    \centering
    \includegraphics[width=0.75\textwidth]{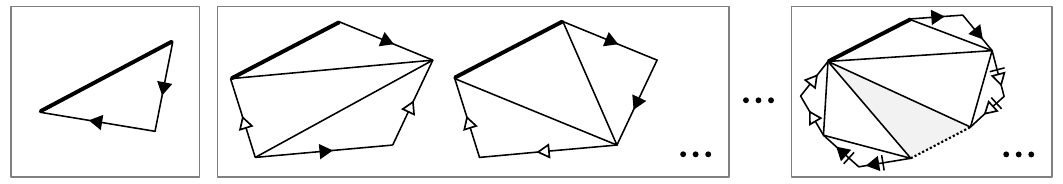}
    \caption{%
    Some examples of $g$-spines; pairs of edges with matching arrows are identified together.
    Left: The only $1$-spine is the one-triangle M\"obius band.
    Centre: There are five different 2-spines~\cite[p.~84]{JacoRubinstein2006};
    here we see the only orientable $2$-spine and one of the four non-orientable $2$-spines.
    Right: An example of a non-orientable $g$-spine for arbitrary $g>4$.}
    \label{fig:g-spines}
\end{figure}

By an Euler characteristic argument, Jaco and Rubinstein show that any $g$-spine has $3g-2$ interior edges.
They then define a \textit{minimal layered handlebody of genus $g$} to be any triangulation that is
built from $3g-2$ tetrahedra by layering over each interior edge of a $g$-spine.
This triangulation is labelled $\mathcal{T}_0$ in the JR-construction~\cite[pp.~83--84]{JacoRubinstein2006}.

Finally, for $t\geqslant 1$, a layered handlebody $\mathcal{T}_t$ of genus $g$, with $t$ layers, is defined inductively,
adding one tetrahedron at a time, by layering over a boundary edge of $\mathcal{T}_{t-1}$.

\subsubsection{Using layered handlebodies}
The following lemma, as stated by Jaco and Rubinstein, is central to the understanding of layered triangulations.
For more on the connectivity of the triangulation complex, see~\cite{Harer86, Hatcher91, Mosher94}.
\begin{lemma}[{\cite[Lemma~3.4]{JacoRubinstein2006}}]
    Any two distinct one-vertex triangulations of a closed orientable surface are related by a finite sequence of flips.
\end{lemma}

Since flips can be realised by layering, the boundary of a genus-$g$ handlebody can always be adjusted to match any other one-vertex triangulation.
This is how Jaco and Rubinstein argue that
any one-vertex triangulation on the boundary of a genus-$g$ handlebody can be extended to a layered triangulation of the handlebody~\cite[Corollary~9.4]{JacoRubinstein2006},
and hence that every closed orientable 3-manifold admits a layered triangulation~\cite[Theorem~10.1]{JacoRubinstein2006}.

When it comes to constructing explicit triangulations, this method begins with two pieces:
a layered handlebody of genus $g$, and the desired genus-$g$ boundary triangulation.
The task is then to match these pieces together by layering.

In the context of Dehn filling, where the desired boundary triangulation appears on a torus boundary of a 3-manifold,
Gu\'eritaud and Schleimer~\cite{GueritaudSchleimer2010} described a way to build the layered solid torus directly onto the 3-manifold boundary.
Their construction builds a complex homeomorphic to $T^2\times\left[0,1\right]$, before identifying the exposed faces at $T^2\times\{1\}$ with a fold.
This fold corresponds to the $g$-spine of the layered solid torus.

The purpose of this paper is to extend this idea to build genus-$g$ layered handlebodies directly onto the boundary of a compact orientable $3$-manifold.
In particular, we can turn a Heegaard splitting into a $3$-manifold triangulation by starting with one layered handlebody,
and attaching a second handlebody by building directly onto the boundary of the first.

Whereas Gu\'eritaud and Schleimer's construction exactly reverses the JR-construction for $g=1$, our construction for $g>1$ allows more flexibility.
In particular, letting $\Sigma_g$ denote the closed orientable surface of genus $g$,
to exactly reverse the JR-construction, the $\Sigma_g\times\left[0,1\right]$ complex
would need to have a triangulation on the $\Sigma_g\times\{1\}$ boundary with a reflective symmetry, to allow for a single fold across a central edge.
For context, in genus $2$, there are nine possible triangulations (up to orientation preserving homeomorphism), of which only five admit the appropriate symmetry.

Our construction instead focuses on ensuring the position of each 2-handle by folding to make the appropriate curve homotopically trivial.
In particular, we may perform some folds, then continue layering before performing subsequent folds.
This means that we might construct a genus-$g$ handlebody that does not contain a $g$-spine at its centre.
It is in this sense that we consider our construction to be a generalisation of the JR-construction.

Importantly, our construction also gives a valid triangulation of a handlebody, and it retains the one-vertex and bounded cutwidth properties.
These properties are discussed carefully in due course.

\subsection{Pinched manifolds}\label{ssec:pinched-manifolds}

When working with triangulations, and particularly when allowing for invalid vertices, it is common to encounter manifolds that are ``pinched'' at a vertex.
This notion is often used in an informal way, without giving a rigorous definition.
However, since pinched manifolds play a crucial role in our work, we introduce a definition here.

Let $M$ be either a closed surface or a $3$-manifold with invalid points. Assume that the link of any invalid point is a genus-$0$ surface with two or more boundary circles.
We say that $M$ contains a \textit{pinched} manifold $P$ if $P$ is immersed in $M$ such that
all self-intersections are isolated \textit{non-crossing} (defined momentarily) intersection points;
for any point $x\in M$ at which there is a self-intersection, we say that $P$ is \textit{pinched at} $x$.

The meaning of ``non-crossing'' is fairly intuitive, but a precise definition requires some care.
Fix any particular point $x\in M$ at which two or more points $p_1,\ldots,p_k\in P$ intersect, and let $U$ be a small regular neighbourhood of $x$ in $M$.
Because the self-intersection points are all isolated, we can choose $U$ to be small enough so that the pre-image $f^{-1}(f(P)\cap U)$ is
a disjoint union of ball neighbourhoods (of the appropriate dimension) $Z_1,\ldots,Z_k\subseteq P$ of $p_1,\ldots,p_k$, respectively
(in particular, $x$ is the only self-intersection point inside $U$, and hence $U$ witnesses the fact that $x$ is isolated).
For each $i\in\{1,\ldots,k\}$, the ball $Z_i$ is embedded in the neighbourhood $U$ (because $f$ is a local embedding), and $f(Z_i)$ separates $U$ into two pieces.
We say that we have a \textit{non-crossing} intersection at $x$ if for all $i,j\in\{1,\ldots,k\}$, $i\neq j$, $f(Z_j)-\{x\}$ lies entirely on one side of $f(Z_i)$ in $U$;
in particular, note that a non-crossing intersection cannot be a transverse intersection.
In the case where $x$ is not an invalid point (in other words, $U$ is a ball), observe that we can remove a non-crossing self-intersection at $x$ via a small perturbation inside $U$.
Figure~\ref{fig:nonCrossing} illustrates these ideas for the case where $P$ is $1$-dimensional and $M$ is $2$-dimensional.

\begin{figure}[htbp]
\centering
\begin{tikzpicture}
\node[inner sep=0pt] at (0,0) {
	\includegraphics[scale=0.8]{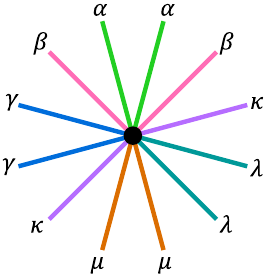}
};
\node[inner sep=0pt] at (6,0) {
	\includegraphics[scale=0.8]{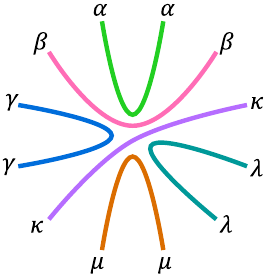}
};
\end{tikzpicture}
\caption{Left: An example of a collection of arcs that meet at an isolated non-crossing intersection point inside a small disc $U$.
Right: The result of removing the non-crossing intersection via a small perturbation inside $U$.}
\label{fig:nonCrossing}
\end{figure}

In this paper, we are mostly interested in the case where $M$ is the underlying topological space of some one-vertex triangulation $\mathcal{T}$.
In such a case, pinched manifolds arise in a natural combinatorial way via manifolds that are pinched at the vertex of $\mathcal{T}$.
For our purposes, we will usually encounter such pinched manifolds in one of the following two scenarios:
\begin{itemize}
\item
Suppose $\mathcal{T}$ is a one-vertex triangulation of a surface $S$.
In Section~\ref{ssec:pinched-diagrams}, we will see that a \emph{pinched filling diagram} is distinguished by the fact
that it is a system of circles in $S$ that is pinched at the vertex of $\mathcal{T}$.
\item
Suppose $\mathcal{T}$ is a one-vertex $3$-triangulation whose vertex $v$ has vertex link given by a genus-$0$ surface with at least two boundary components (hence, $v$ is invalid).
The boundary triangles of $\mathcal{T}$ will triangulate a closed surface that is pinched at $v$.
\end{itemize}

In the latter scenario, our algorithm must handle the following complication.
Although $\mathcal{T}$ is a one-vertex $3$-triangulation, the $2$-triangulation $B$ induced by the boundary triangles of $\mathcal{T}$ will \emph{not} be one-vertex;
indeed, the vertices of $B$ are precisely the points of $B$ that are all pinched together at $v$.
This can be seen just by considering vertex links:
the link of $v$ in $\mathcal{T}$ is a surface with at least two boundary circles, and each such circle forms the link of a vertex of $B$.
For our purposes, the key implication is that there are two types of closed curves in $\partial\mathcal{T}$ intersecting $v$:
\begin{itemize}
\item
closed curves in $B$ intersecting a vertex; and
\item
arcs in $B$ whose endpoints lie on distinct vertices of $B$.
\end{itemize}
The latter arcs give closed curves in $\partial\mathcal{T}$ because their endpoints are pinched together at $v$.
We will see that it is crucial for our main algorithm to distinguish these two types of curves.

\section{Filling diagrams}\label{sec:HeegaardDiagrams}
We begin by defining topological filling diagrams, and describing their relationship to Heegaard diagrams.
We then introduce pinched filling diagrams, and justify their equivalence to topological filling diagrams.
The remainder of this section is then dedicated to introducing \textit{combinatorial filling diagrams},
which consist of a one-vertex triangulation of a compact orientable $3$-manifold,
along with combinatorial data encoding the position of attaching circles with respect to the boundary triangulation.

\subsection{Topological filling diagrams}\label{sec:standard-diagram}
Given a closed orientable surface $\Sigma$ of genus $g$, a \emph{system of attaching circles} $\gamma=\{\gamma_1,\ldots,\gamma_g\}$ on $\Sigma$ is
a collection of $g$ embedded closed curves that are pairwise disjoint, and are non-separating as a multicurve on $\Sigma$.
Define a \emph{topological filling diagram} $D$ to be a compact orientable $3$-manifold $M$,
along with a non-empty union $B=B^1\cup\cdots\cup B^{\,k}$ of boundary components of $M$,
where each boundary component $B^{\,i}$ is marked with a system $\gamma^{\,i}$ of attaching circles;
we write $D=(M,B,\gamma)$, where $\gamma=\gamma^1\cup\cdots\cup\gamma^{\,k}$.
Occasionally, when it is necessary to make a clear distinction from attaching circles in some other context (such as Heegaard diagrams, as described in
Setting~\ref{settingB}), we instead refer to the attaching circles in a topological filling diagram as \emph{filling curves}.

For simplicity, we focus on the case where $M$ has exactly one boundary component, in which case $B=\partial M$.

Recall from Setting~\ref{settingB} that a Heegaard diagram consists of a closed orientable surface $\Sigma$ with two systems $\beta$ and $\gamma$ of attaching circles;
this specifies a closed $3$-manifold $M^\ast$ because $\beta$ and $\gamma$ specify how to glue two handlebodies together along $\Sigma$.
From such a Heegaard diagram, we can easily obtain a topological filling diagram $D$ that specifies the same $3$-manifold $M^\ast$:
simply take $D=(M,\partial M,\gamma)$, where $M$ is a handlebody obtained by filling $\Sigma$ according to $\beta$.

A topological filling diagram can be used to build the corresponding filled manifold as described in the following `folklore' algorithm.
For the case of Heegaard diagrams see, for example, the books by Rolfsen~\cite[Chapter~9]{Rolfsen} or Schultens~\cite[Chapter~6]{Schultens2014Book}.

\begin{algorithm}[Topological filling algorithm]\label{algm:standard-algorithm}
As input, take the filling diagram $D=(M,B,\gamma)$, where $M$ is a compact orientable $3$-manifold with
a genus-$g$ boundary $B$ marked with the system of attaching circles $\gamma$, as well as a genus-$g$ handlebody $H$.
\begin{enumerate}
    \item\label{step:standard-1}
    Remove a collar neighbourhood of the discs in $H$ bounded by the meridians of its handles.
    These pieces are 2-handles.
    Set $H'=H-\{\text{2-handles}\}$ and note that $H'$ is homeomorphic to a 3-ball with $2g$ marked discs.
    \item\label{step:standard-2}
    Attach the 2-handles to $M$, by identifying the meridian curves from $H$ to the attaching circles $\gamma$ in $\partial M$.
    Set $M'=M\cup \{\text{2-handles}\}$ and note that the boundary of $M$ is homeomorphic to $S^2$ with $2g$ marked discs.
    \item Glue $H'$ to the $S^2$ boundary of $M'$, ensuring that the marked discs are glued back to each other,
    as they were in $H$.
    This is always possible due to Alexander's lemma (also often known as the Alexander trick);
    see~\cite[Lemma~2.5.3]{Schultens2014Book} for a statement of this standard lemma.
    \item Return the filled $3$-manifold $M(\gamma)$.
\end{enumerate}
\end{algorithm}

\subsection{Pinched filling diagrams}\label{ssec:pinched-diagrams}
Let us introduce an adaptation of Algorithm~\ref{algm:standard-algorithm},
which is more directly comparable to the triangulation built by our combinatorial algorithm.
We refer to this as the \textit{pinched filling algorithm}.

As in Algorithm~\ref{algm:standard-algorithm}, let $H$ be a genus-$g$ handlebody.
Suppose $D^2\times\left[-1,1\right]$ is a collar neighbourhood of a properly embedded disc in $H$, with a marked point $(p,0)\in\partial D^2\times\{0\}$.
We define the corresponding \textit{wedge neighbourhood} to be a quotient of the collar neighbourhood,
found by collapsing the segment $\{p\}\times\left[-1,1\right]$ to the single point $(p,0)$.

Intuitively, we think of the wedge neighbourhoods as ``pinched'' versions of the $2$-handles that appear in
steps~\ref{step:standard-1} and~\ref{step:standard-2} of Algorithm~\ref{algm:standard-algorithm}.
That is, instead of removing collar neighbourhoods of the discs in $H$ bounded by meridians,
the pinched filling algorithm removes wedge neighbourhoods (as in Figure~\ref{fig:wedge-nbhd}) and sets $\hat{H}=H-\{\text{wedges}\}$.
Note that $\hat{H}$ is homeomorphic to a 3-ball with $2g$ marked discs, each with a marked point.

    \begin{figure}[htbp]
        \centering
        \includegraphics[width=0.8\textwidth]{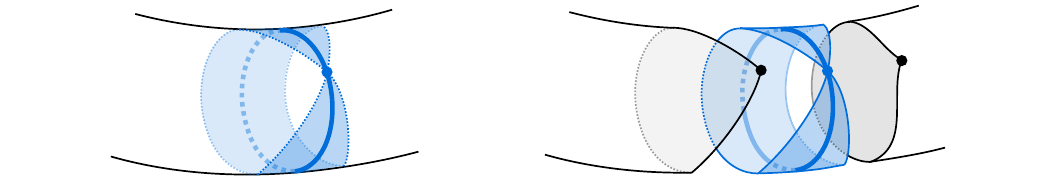}
        \caption{A wedge neighbourhood of a meridian disc.
        Removing a wedge from $H$ leaves two marked discs on $\partial H$, each with a marked point.}
        \label{fig:wedge-nbhd}
    \end{figure}

On the boundary of $M$, before attaching the wedge pieces, we isotope each attaching circle so that
they all intersect at a single point (but remain pairwise disjoint away from this common point).
This is always possible since the complement of the attaching circles in $\partial M$ is path-connected.
The result is that the system of attaching circles is no longer embedded in $\partial M$, but rather pinched at a single point (as defined in Section~\ref{ssec:pinched-manifolds}).
We denote the set of pinched attaching circles by $\hat{\gamma}$.

Although there are, up to isotopy, many different ways to pinch the attaching circles together,
recall from the definition of pinching that the intersections between the attaching circles must be \emph{non-crossing}.
This means that we can recover a unique (up to isotopy) embedded system of attaching circles via a small perturbation.
The upshot of this is that the choice of pinching will not change the final $3$-manifold that we construct.

The wedges from $H$ can be attached to the pinched attaching circles on $M$, with the point of each wedge meeting the pinched point of the attaching circle.
We set $\hat{M}=M\cup\{\text{wedges}\}$ and claim that the boundary of $\hat{M}$ is
a pinched 2-sphere (as defined in Section~\ref{ssec:pinched-manifolds});
see Proposition~\ref{prop:algm-equiv-S=P} below.

    \begin{figure}[htbp]
        \begin{subfigure}{0.45\textwidth}
            \centering
            \includegraphics[width=\linewidth]{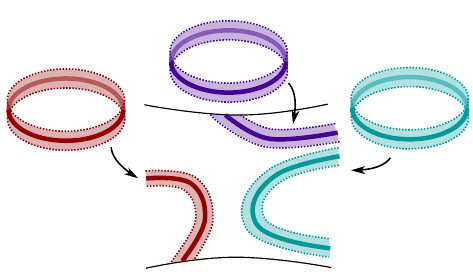}
            \caption{Attaching 2-handles to attaching circles in the standard algorithm.}
            \label{subfig:step-2-standard}
        \end{subfigure}
        \hfill
        \begin{subfigure}{0.45\textwidth}
            \centering
            \includegraphics[width=\linewidth]{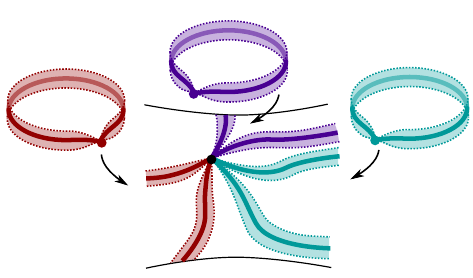}
            \caption{Attaching wedges to pinched attaching circles in the pinched algorithm.}
            \label{subfig:step-2-pinched}
        \end{subfigure}
    \caption{Comparing Step~\ref{step:standard-2} of Algorithm~\ref{algm:standard-algorithm} with
    the corresponding step in the pinched filling algorithm.}
    \label{fig:standard-vs-pinched}
    \end{figure}

To conclude, we glue $\hat{H}$ to the boundary of $\hat{M}$, ensuring that the marked discs and marked points are glued to each other, as they were in $H$.
Again, this is possible due to Alexander's lemma~\cite[Lemma~2.5.3]{Schultens2014Book}.

In summary, the pinched filling algorithm takes a pinched filling diagram
$\hat{D} = ( M, B, \hat{\gamma} )$, where the $\hat{\gamma}$ curves have been isotoped to meet at a point,
and proceeds as in Algorithm~\ref{algm:standard-algorithm}, except that it uses wedges instead of 2-handles.

\begin{proposition}\label{prop:algm-equiv-S=P}
The boundary of $\hat{M}=M\cup\{\text{wedges}\}$ is a pinched 2-sphere and
    the pinched filling algorithm builds the same 3-manifold $M(\gamma)$ as Algorithm~\ref{algm:standard-algorithm}.
\end{proposition}
\begin{proof}
    Recall that $M'=M\cup \{\text{2-handles}\}$ is the object obtained in Step~\ref{step:standard-2}
    of Algorithm~\ref{algm:standard-algorithm}, and that its boundary is a 2-sphere.

    We show that we can deformation retract $M'$ to $\hat{M}$ inside $M(\gamma)$, without changing the topology of $M(\gamma)$;
    Figure~\ref{fig:proof-S=P} illustrates the idea of the proof in the genus $2$ case.
    Note that since the algorithm fills the 2-sphere component with a 3-ball, we may assume that $M(\gamma)-M'$ is simply connected.

    First, convert each 2-handle inside $M'$ into a wedge by contracting a line segment on the interval boundary of the 2-handle to a point.
    This brings together two points on $\partial M'$ by pinching inwards into $M'$, resulting in a pinched 2-sphere boundary.

    Note that there is a path on the pinched 2-sphere boundary connecting any two distinct wedge points.
    Contract $g-1$ such paths so that all wedge points are pulled together.
    This yields $\hat{M}$, which is the pinched version of $M'$, and its boundary is still a pinched 2-sphere.

    Since $\hat{M}$ exists inside $M(\gamma)$, the changes we made to the algorithm are valid.
    Moreover, the changes do not affect the topology of $M(\gamma)$, so the pinched version of the algorithm also builds the desired $3$-manifold.
    \qedhere

    \begin{figure}[htbp]
    \centering
        \begin{subfigure}[t]{0.49\textwidth}
            \centering
            \includegraphics[width=0.75\linewidth]{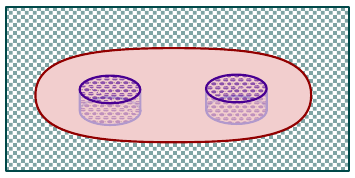}
            \caption{The output of Step~\ref{step:standard-2} in Algorithm~\ref{algm:standard-algorithm}. }
            \label{subfig:proof-S=P-a}
        \end{subfigure}
        \begin{subfigure}[t]{0.49\textwidth}
            \centering
            \includegraphics[width=0.75\linewidth]{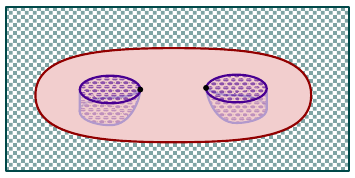}
            \caption{Pinching 2-handles to wedges through the 2-handles.}
            \label{subfig:proof-S=P-b}
        \end{subfigure}
        \newline
        \begin{subfigure}[t]{0.49\textwidth}
            \centering
            \includegraphics[width=0.75\linewidth]{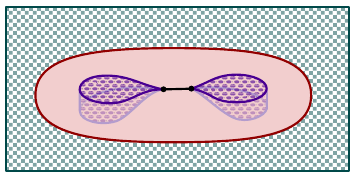}
            \caption{Contracting wedge points along a path in the boundary of the 3-ball.}
            \label{subfig:proof-S=P-c}
        \end{subfigure}
         \begin{subfigure}[t]{0.49\textwidth}
            \centering
            \includegraphics[width=0.75\linewidth]{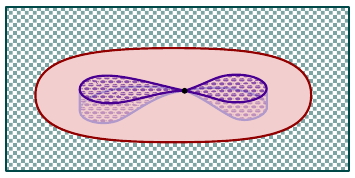}
            \caption{The pinched version of~(a), obtained by contracting paths within $M(\gamma)$.}
            \label{subfig:proof-S=P-d}
        \end{subfigure}
    \caption{An example when $\partial M$ is genus $2$, justifying the equivalence of
    Algorithm~\ref{algm:standard-algorithm} and the pinched filling algorithm. The plain red region represents $M$, which has non-trivial topology in general.
    The dotted purple regions represent the 2-handles and wedge pieces, and the checkered green region represents the 3-ball that is glued to the boundary of $M'$.}
        \label{fig:proof-S=P}
    \end{figure}
\end{proof}

\subsection{Combinatorial filling diagrams}
Here we introduce the \textit{combinatorial filling diagram},
which is used as input for our main algorithm (Algorithm~\ref{algm:main-combinatorial}).
Later, in Theorem~\ref{thm:main-alg}, we prove that the construction of a $3$-manifold from
a combinatorial filling diagram is essentially a triangulated version of the pinched filling algorithm.

Let $M$ be a compact orientable $3$-manifold whose boundary is a genus-$g$ surface $B$, as in Algorithm~\ref{algm:standard-algorithm}.
Our goal is to attach a genus-$g$ handlebody $H$ along $B$, according to a system of attaching circles.
In the combinatorial setting, we require $M$ and $H$ to be equipped with one-vertex triangulations $M_\Delta$ and $H_\Delta$, respectively.
While the triangulation $M_\Delta$ is given as part of the input, we construct $H_\Delta$ throughout Algorithm~\ref{algm:main-combinatorial} by building directly onto the boundary of $M_\Delta$;
in other words, we obtain $H_\Delta$ as the complement of $M_\Delta$ in the triangulation constructed by Algorithm~\ref{algm:main-combinatorial}.

\subsubsection{Rooted normal curves}
Since the boundary of $M_\Delta$ inherits a one-vertex triangulation,
it is helpful to isotope the attaching circles so that they intersect nicely with the edges of the boundary triangulation.
Since the attaching circles are essential, we may assume that they are in \textit{normal position} with respect to the triangulation, meaning that the curves
are disjoint from the vertex, intersect the edges transversely, and intersect the faces in \emph{normal arcs} (see Definition~\ref{def:arcs} below);
this standard fact is proven, for example, in Matveev's book~\cite[Theorem~3.2.2 and Remark~3.2.3]{Matveev2007}.
However, for input into our algorithm, we ask that the attaching circles be \textit{rooted}:
that is, we isotope the curves so that they are all pinched at the vertex of $M_\Delta$.

\begin{definition}\label{def:arcs}
    A \textit{normal arc} is a simple curve that enters a face from one edge and leaves from a different edge, thus cutting off one of the three corners. A \textit{rooted arc} is a simple curve that runs from one corner, across the face, and leaves via the opposite edge.

Consider any triangle $f$ in $\partial M_\Delta$.
Two normal or rooted arcs in $f$ have the same \textit{type} if they are related by a \textit{normal isotopy}
(an ambient isotopy that preserves every vertex, edge and triangle of $\partial M_\Delta$);
equivalently, the arcs have the same type if and only if they intersect the same edges of $f$.
Figure~\ref{fig:n-r-coords} shows four of the six possible arc types in $f$:
all three normal arc types on the left, and one of the three rooted arc types in the centre.
\end{definition}

\begin{definition}
    The number of normal and rooted arcs in a triangle $\Delta$ are encoded using \textit{(rooted normal) arc coordinates}. If $e_0,\,e_1,\,e_2$ denote the edges of $\Delta$, we let $n_i$ be the number of normal arcs opposite the edge $e_i$ and let $r_i$ be the number of rooted arcs through the edge $e_i$.
    See Figure~\ref{fig:n-r-coords}, right, for an explicit example of a system of arcs corresponding to a choice of arc coordinates.
\end{definition}

\begin{figure}[htbp]
    \centering
    \includegraphics[width=0.7\textwidth]{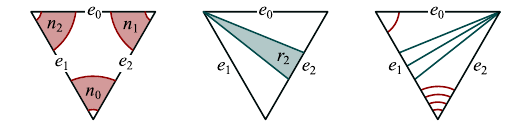}
    \caption{Left: Normal arc types and their coordinates. Centre: An example of a rooted arc type and its coordinate. Right: A system of arcs with coordinates $n_0=4$, $r_1=3$, $n_2=1$, and all other coordinates zero.}
    \label{fig:n-r-coords}
\end{figure}

\begin{definition}
An embedded closed curve $\gamma$ in a $2$-dimensional one-vertex triangulation is in \emph{rooted normal position} if either:
\begin{itemize}
\item $\gamma$ coincides with an edge of the triangulation, in which case we say that $\gamma$ is \emph{resolved}; or
\item $\gamma$ consists of exactly two rooted arcs together with some finite number (possibly zero) of normal arcs, in which case we say that $\gamma$ is \emph{unresolved}.
\end{itemize}
A collection of curves is in \emph{rooted normal position} if: \emph{(a)} each individual curve is in rooted normal position,
\emph{(b)} the intersections of the curves at the vertex are all non-crossing,
and \emph{(c)} the curves are all pairwise disjoint away from the vertex.
\end{definition}

In Proposition~\ref{prop:rooting} below, we show that on a one-vertex triangulation of a genus-$g$ surface,
we can always isotope a system of attaching circles to be in rooted normal position.
Before we prove this, it is worth commenting on why we develop this notion of rooted normal curves,
rather than using the well-studied and widely-used notion of normal curves.

\begin{remark}\label{rem:rootedMotivation}
Our main algorithm relies on the following useful property of rooted normal curves:
the number of times that such a curve intersects the (interiors of) edges of the triangulation can be reduced to zero, thereby yielding a resolved curve.
Our algorithm is then able to ignore resolved curves while it deals with all the other unresolved curves.
An earlier version of our algorithm attempted to apply a similar philosophy directly to normal curves.
This was surprisingly complicated and inelegant, and we found that rooted normal curves solved many of the problems that we had with normal curves;
in particular, our proof of Proposition~\ref{prop:PR-terminates} in Section~\ref{ssec:PR} relies crucially on the curves being rooted,
and provides a glimpse of the complications that arise if one tries to use normal curves instead.
\end{remark}

\begin{proposition}\label{prop:rooting}
    Let $B_\Delta$ be a closed orientable genus-$g$ surface equipped with a one-vertex triangulation.
    A system of attaching circles on $B_\Delta$, in normal position with respect to the triangulation, can be isotoped to be in rooted normal position.
\end{proposition}

\begin{proof}
    Recall that a valid system of attaching circles consists of $g$ non-intersecting curves that are collectively non-separating.
    By cutting $B_\Delta$ along all attaching circles, we obtain a surface homeomorphic to $S^2$ with $2g$ discs removed.
    In particular, the complement of the attaching circles in $B_\Delta$ is path-connected.
    Moreover, disjoint paths connecting each curve to the vertex can be embedded in $B_\Delta$ simultaneously.

    In fact, these paths can be positioned with respect to the triangulation in such a way that for each attaching circle $\gamma$,
    there is at least one triangle in which a path can be drawn between a normal arc belonging to $\gamma$ and one of the corners of the triangle.
To see why, suppose there is some attaching circle $\gamma$ for which this does not hold;
that is, suppose every corner of every triangle is separated from $\gamma$ by at least one arc of some other attaching circle.
Cutting $B_\Delta$ along all attaching circles other than $\gamma$ would yield a disconnected surface with at least two components:
one containing the vertex, and one containing $\gamma$.
Since a valid set of attaching circles is non-separating, this gives a contradiction.

Now, consider a triangle that admits a path between an arc of $\gamma$ and one of its corners.
Up to symmetry, there are two possible configurations of arcs in this triangle, depending on whether the path connects $\gamma$ to the corner it cuts off,
or to one of the corners across the triangle (see Figure~\ref{fig:rooting1}).
For each attaching circle $\gamma$, we choose one normal arc $\alpha$, together with a path connecting $\alpha$ to a corner of the triangle containing $\alpha$,
and then isotope along this path so that $\gamma$ becomes rooted.
Any single triangle can have attaching circles rooted to at most one of its corners;
we call this corner (if it exists) the \textit{root corner}.

    \begin{figure}[htbp]
        \centering
        \includegraphics[width=0.5\textwidth]{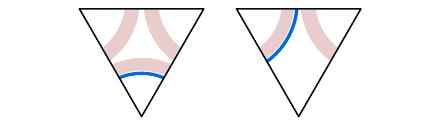}
        \caption{For each attaching circle $\gamma$ in normal position, there exists a triangle exhibiting one of these two configurations.
        In particular, there is a normal arc belonging to $\gamma$ that appears as either:
        the innermost normal arc in one of the three corners of the triangle (left);
        or the outermost normal arc in one of the three corners of a triangle where
        one of the other two corners of the triangle has no normal arcs (right).
        Any additional normal arcs must appear in the shaded regions.}
        \label{fig:rooting1}
    \end{figure}

    This process introduces arcs that are no longer normal with respect to the triangulation, as seen in Figure~\ref{fig:rooting2}.
    In particular, we see arcs that meet a corner at one end and cross an adjacent edge at the other,
    forming bigons between the edges of the triangulation and the rooted attaching circle.

    \begin{figure}[htbp]
        \centering
        \includegraphics[width=0.5\textwidth]{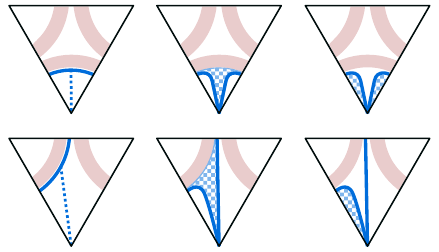}
        \caption{
In either of the cases shown in Figure~\ref{fig:rooting1}, there is a path
(shown here in each of the leftmost triangles as a dotted blue line) connecting the distinguished normal arc to the bottom corner of the triangle. Checkered regions indicate isotopy: we first isotope the arc to the bottom corner (centre), which introduces one or two bigons; then bigons are removed by isotopy across an edge (right).}
        \label{fig:rooting2}
    \end{figure}

    To remove bigons, we isotope across the edge into the adjacent triangle.
    The isotoped arc either leaves the adjacent triangle via the edge \emph{opposite} the root corner
    (as in Figure~\ref{fig:rooting3}, left),
    or it leaves via the edge \emph{adjacent} to the root corner (as in Figure~\ref{fig:rooting4}, left).
    In the first of these cases, the result is a rooted arc (as in Figure~\ref{fig:rooting3}, right).
    In the second case, we have created a new bigon (as in Figure~\ref{fig:rooting4}, right).
    When this happens, we continue isotoping the curve across edges until all such bigons are removed.
    Note that this process must terminate because the attaching circle cannot fully encircle the vertex.

\begin{figure}[htbp]
        \centering
        \begin{subfigure}{\textwidth}
            \centering
            \includegraphics[width=0.48\textwidth]{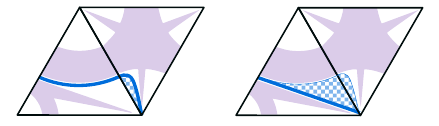}
        \caption{If the continuation of the curve leaves the adjacent triangle via
        the edge opposite the root corner, then we obtain a rooted arc.}
        \label{fig:rooting3}
        \end{subfigure}
        \newline
        \begin{subfigure}{\textwidth}
            \centering
            \includegraphics[width=0.48\textwidth]{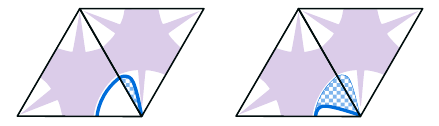}
        \caption{If the continuation of the curve leaves the adjacent triangle via the other edge meeting the root corner,
        then we have created a new bigon in this triangle and we continue recursively.
        }
        \label{fig:rooting4}
        \end{subfigure}
    \caption{The two possible scenarios that arise when we isotope a bigon across an edge into the adjacent triangle.
    The purple shaded regions indicate the possible positions of other arcs in each triangle and the checkered regions
    again indicate the region through which we isotope the arcs of interest.}
\end{figure}

    Having rooted each distinct attaching circle in one place,
    there are two ``ends'' where the attaching circle leaves and then returns to the vertex. The process of removing these bigons terminates in one of three ways.
    First, the process may terminate when both ends of the rooted attaching circle form a rooted arc in different triangles, like in Figure~\ref{subfig:rooting5a}.
    The second possibility is similar, except that both ends appear as rooted arcs in
    the same triangle, from the same corner, as in Figure~\ref{subfig:rooting5b}.
    Finally, the two ends of the rooted attaching circle may be rooted at
    two distinct corners in the same triangle, without leaving the triangle, as in Figure~\ref{subfig:rooting5c}.
    In this last case, the attaching circle is isotopic to an edge of the triangle, and is therefore resolved.

    \begin{figure}[htbp]
        \centering
        \begin{subfigure}{0.32\textwidth}
        \centering
        \includegraphics[width=0.8\textwidth]{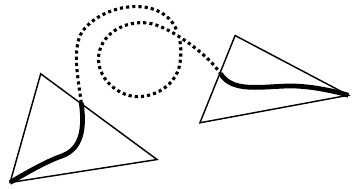}
        \caption{Each of the two ends are rooted in distinct triangles.}
        \label{subfig:rooting5a}
        \end{subfigure}
        \begin{subfigure}{0.32\textwidth}
        \centering
        \includegraphics[width=0.8\textwidth]{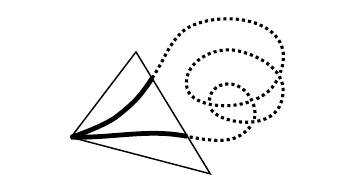}
        \caption{Both ends are rooted in the same triangle, in the same corner.}
        \label{subfig:rooting5b}
        \end{subfigure}
        \begin{subfigure}{0.32\textwidth}
        \centering
        \includegraphics[width=0.8\textwidth]{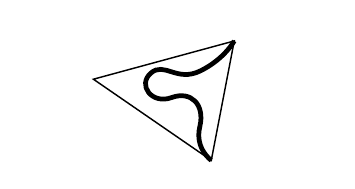}
        \caption{Both ends are rooted in the same triangle but at different corners.}
        \label{subfig:rooting5c}
        \end{subfigure}
        \caption{After removing bigons, the two ``endpoints'' of a rooted
        attaching circle may appear in one of three configurations.}
        \label{fig:rooting5}
    \end{figure}

    Arcs where no isotopy was required remain in normal position.
    Hence, at the end of this process the attaching circles are positioned such that they either
    coincide with edges of the triangulation, or they consist entirely of normal and rooted arcs.
    Thus, they are in rooted normal position, as claimed.
\end{proof}

The collection of rooted attaching circles form a bouquet, in the graph-theoretical sense, so we refer to the set of
rooted attaching circles as a \textit{filling bouquet}, and individual curves as \textit{(filling) petals}.
When a filling bouquet is understood to be valid, the petals all meet at the vertex in pairwise non-crossing intersections,
which means that we can recover a collection of disjoint embedded curves via a small perturbation of the petals (recall Section~\ref{ssec:pinched-manifolds});
thus, we will often abuse terminology and say that the petals themselves are disjoint.

\subsubsection{Describing curves via edge weights}\label{sec:weights=curves}

To efficiently encode our rooted systems of attaching circles, we count the number of intersections with each edge, and hence assign a ``weight'' or ``coordinate'' to each edge.
Such edge coordinates are well-known in the classical setting of \textit{normal curves} (i.e., curves consisting entirely of normal arcs);
for instance, see~\cite{Bell2015Thesis,EricksonNayyeri2013,SSS2002Curves,SSS2003String,SSS2008Dehn,Stefankovic2005Thesis}.
Not every vector of these edge coordinates corresponds to a normal curve, which opens the possibility of using such coordinates to encode more general objects;
indeed, some such possibilities have already been studied in contexts unrelated to this paper~\cite{GSC2021Coordinates,GSC2021Conformal}.
In this section, we explain how to use edge weights to encode our rooted normal curves;
to our knowledge, this specific application of edge weights is new to the literature.

Let $B_\Delta$ denote a one-vertex triangulation of a closed orientable surface of genus $g$,
and suppose we have a system of attaching circles in rooted normal position in $B_\Delta$.
We define the \textit{weight} of an edge $\textbf{\textit{e}}$ in $B_\Delta$ to be the number of transverse intersections between all attaching circles and
the interior of $\textbf{\textit{e}}$ (this does not include any intersections at the vertex).
An edge $\textbf{\textit{e}}$ in $B_\Delta$ has \textit{maximal weight} if there is no edge in $B_\Delta$ with strictly higher weight than $\textbf{\textit{e}}$.
The \textit{total weight} of $B_\Delta$ is the sum of all edge weights.
We use edge weights to encode the unresolved attaching circles.

Recall that a triangle $\Delta$ with edges $\textbf{\textit{e}}_0,\,\textbf{\textit{e}}_1,\,\textbf{\textit{e}}_2$ contains normal arcs $n_0,\,n_1,\,n_2$ and rooted arcs $r_0,\,r_1,\,r_2$.
Denote the weight of edge $\textbf{\textit{e}}_i$ by $w_i$.
Observe that $w_i=r_i+n_{i+1}+n_{i-1}$, where subscripts are taken to be in $\mathbb{Z}_3$.

\begin{proposition}[Edge weights determine rooted normal curves]\label{prop:weights=curves}
    There is at most one choice of rooted normal arc coordinates that is compatible with the edge weights around a face $\Delta$.

Indeed, assuming without loss of generality that $w_0 \leqslant w_1 \leqslant w_2$, precisely one of the following must hold:
\begin{enumerate}[label={(\alph*)}]
\item\label{case:a:rooted}
$w_0+w_1 < w_2$, in which case the arc coordinates in $\Delta$ must be
\[
n_0 = w_1,\quad
n_1 = w_0,\quad
n_2 = 0,\quad
r_0 = r_1 = 0,\quad
r_2 = w_2 - w_0 - w_1.
\]
\item\label{case:b:normal}
$w_0+w_1\geqslant w_2$ and $w_0+w_1+w_2$ is even, in which case the arc coordinates in $\Delta$ must be
\[
n_0 = \frac{
w_1 + w_2 - w_0
}{2},\quad
n_1 = \frac{
w_0 + w_2 - w_1
}{2},\quad
n_2 = \frac{
w_0 + w_1 - w_2
}{2},\quad
r_0=r_1=r_2=0.
\]
\item\label{case:c:incompatible}
$w_0+w_1\geqslant w_2$ and $w_0+w_1+w_2$ is odd, in which case there is no compatible choice of arc coordinates in $\Delta$.
\end{enumerate}
\end{proposition}

Before proving this proposition, we collect some useful observations and basic facts about arc coordinates and their restrictions. Throughout this section we continue to assume that subscripts are in $\mathbb{Z}_3$.

\begin{lemma}[Rooted arc constraints]
    Only one of the three rooted arc coordinates may be positive in a given triangle,
    and the normal coordinate corresponding to the positive root coordinate must be 0.
    That is, if $r_i>0$ then we must have $n_i=0$ and $r_{i\pm 1}=0$, as illustrated in Figure~\ref{fig:fundamental-observation}.
\end{lemma}

\begin{figure}[htbp]
        \centering
        \includegraphics[width=0.72\textwidth]{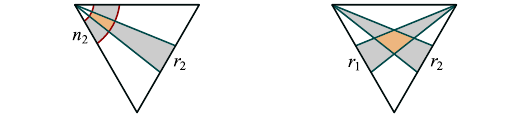}
        \caption{When a rooted arc is present, meaning $r_i>0$, then there can be no normal arcs in that corner, so $n_i=0$, and the other rooted arc coordinates $r_{i\pm 1}$ must also be 0.}
        \label{fig:fundamental-observation}
    \end{figure}

    \begin{lemma}\label{lem:roots-and-weights}
        If $r_i>0$, then $w_i>w_{i+1}+w_{i-1}$.
    \end{lemma}
    \begin{proof}
        With reference to Figure~\ref{fig:roots-and-weights}, we have
        \[ r_i>0 \implies w_i=r_i+n_{i-1}+n_{i+1}>n_{i-1}+n_{i+1}=w_{i+1}+w_{i-1}.\qedhere\]

    \begin{figure}[htbp]
        \centering
        \includegraphics[width=0.72\textwidth]{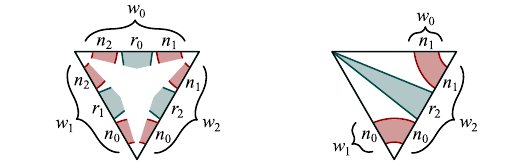}
        \caption{Left: In general, $w_i=n_{i-1}+r_i+n_{i+1}$. Right: If the root coordinate $r_i$ is positive then $w_{i-1}=n_{i+1}$, $w_{i+1}=n_{i-1}$, and $w_i=w_{i+1}+r_i+w_{i-1}$.}
        \label{fig:roots-and-weights}
    \end{figure}
    \end{proof}

    \begin{corollary}\label{cor:max-weight-root}
        If edge $e_i$ has maximal weight, then the only rooted arc coordinate that can possibly be positive is $r_i$.
    \end{corollary}

    \begin{lemma}\label{lem:r0=r1=zero}
    If $w_0\leqslant w_1\leqslant w_2$, then $r_0=r_1=0$.
    \end{lemma}
    \begin{proof}
        If $r_0>0$ then $w_0> w_1+w_2$ by Lemma~\ref{lem:roots-and-weights},
        but this implies that $w_0>w_2$, which is a contradiction to our assumed ordering.
        A similar argument applies to $r_1$.
    \end{proof}

    \begin{lemma}\label{lem:r2-positive}
        If $w_0\leqslant w_1\leqslant w_2$, then  $r_2\geqslant w_2-w_1-w_0$.
    \end{lemma}
    \begin{proof}
        By definition, $r_2=w_2-n_0-n_1$. By the rooted arc constraints, we have $w_1=n_0+n_2$ and $w_0=n_1+n_2$, so this becomes $r_2=w_2+(n_2-w_1)+(n_2-w_0)$. Then by expanding we see that $r_2=w_2+2n_2-w_1-w_0\geqslant w_2-w_1-w_0$.
    \end{proof}

\begin{proof}[Proof of Proposition~\ref{prop:weights=curves}]
Assume, without loss of generality, that ${w_0\leqslant w_1\leqslant w_2}$. Hence, $r_0=r_1=0$ (by Lemma~\ref{lem:r0=r1=zero}) and it remains to be shown which values $r_2,\ n_0,\ n_1$ and $n_2$ can take.
For this, we first note that all possible sets of edge weights fall into one of the following cases, corresponding precisely to the three cases in the proposition statement:
\begin{description}
\item[\ref{case:a:rooted}]
$w_0+w_1 < w_2$,
\item[\ref{case:b:normal}]
$w_0+w_1\geqslant w_2$ and $w_0+w_1+w_2$ is even,
\item[\ref{case:c:incompatible}]
$w_0+w_1\geqslant w_2$ and $w_0+w_1+w_2$ is odd.
\end{description}
In each of these cases, we need to explain why the possible arc coordinates are completely determined by the edge weights.
Specifically, we will show that: in~\ref{case:a:rooted}, there must be at least one rooted arc;
in~\ref{case:b:normal}, the only choice of compatible arc coordinates consists entirely of normal arcs;
and in~\ref{case:c:incompatible}, there is no realisation of the edge weights as compatible rooted and normal arcs.

First consider when $w_0+w_1 < w_2$ (Case~\ref{case:a:rooted}). We have
    \begin{align*}
        w_2-w_1-w_0>0&\implies r_2>0 & \text{(Lemma~\ref{lem:r2-positive})}\\
        &\implies n_2=0 & \text{(Rooted arc constraints)}\\
        &\implies n_0=w_1,\ n_1=w_0 \text{ and } r_2=w_2-n_0-n_1 & \text{(see Figure~\ref{fig:roots-and-weights})}\\
        &\implies r_2=w_2-w_1-w_0. &
    \end{align*}

    Thus, when $w_0+w_1 < w_2$, the rooted normal arc coordinates may be expressed in terms of edge weights as follows:
    \[
    n_0 = w_1,\quad
    n_1 = w_0,\quad
    n_2 = 0,\quad
    r_0 = r_1 = 0,\quad
    r_2 = w_2 - w_0 - w_1.
    \]
This completes the proof for~\ref{case:a:rooted}.

    Now suppose $w_0+w_1\geqslant w_2$ (cases~\ref{case:b:normal} and~\ref{case:c:incompatible}). Recall, by Lemma~\ref{lem:r0=r1=zero}, that $r_0=r_1=0$. Note that $r_2$ must also be 0, since otherwise our inequality would contradict Lemma~\ref{lem:roots-and-weights}. When there are no rooted arcs, as in Figure~\ref{fig:n-r-coords} (left), we see that
    \[ w_0+w_1+w_2=2(n_0+n_1+n_2)\quad \implies \quad w_0+w_1+w_2 \text{ is even.}\]
    In particular, Case~\ref{case:c:incompatible} is incompatible with any valid set of rooted normal arc coordinates.

    When $w_0+w_1+w_2$ is even (Case~\ref{case:b:normal}), we can turn the edge weights into normal arc coordinates in the usual way.
    In detail, for each $i\in\mathbb{Z}_3$ we have
    \[ w_{i}+w_{i+1}+w_{i-1}=2(n_{i}+n_{i+1}+n_{i-1})=2n_{i}+2w_{i}, \quad\text{as in Figure~\ref{fig:n-r-coords} (left).}\]
    This implies that $n_i=\frac{1}{2}(w_{i+1}+w_{i-1}-w_{i})$.
    Thus, in Case~\ref{case:b:normal}, the rooted normal arc coordinates are expressed in terms of edge weights as follows:
    \[
    n_0 = \frac{
    w_1 + w_2 - w_0
    }{2},\quad
    n_1 = \frac{
    w_0 + w_2 - w_1
    }{2},\quad
    n_2 = \frac{
    w_0 + w_1 - w_2
    }{2},\quad
    r_0=r_1=r_2=0.
    \]
This completes the proof of the proposition.
\end{proof}

\begin{remark}\label{rmk:invalid-input}
In general, even if a set of edge weights is compatible with some choice of
rooted normal arc coordinates, the corresponding collection of curves need not form a filling bouquet:
there might be curves formed entirely by normal arcs, which are therefore not rooted;
and there might also be a pair of rooted curves whose intersection (at the vertex) fails the non-crossing condition.
We use the term \textit{valid} to describe a set of edge weights or a collection of filling petals that avoids these features.
\end{remark}

\subsubsection{Defining combinatorial filling diagrams}
Now that we have established the fact that attaching circles can be isotoped to rooted normal position, and that a valid set of edge weights uniquely
determines a collection of attaching circles, we are ready to define a combinatorial filling diagram.
Although our definition is restricted to the case where the input manifold has only one boundary component, this is not necessary in principle.

\begin{definition}[Combinatorial filling diagram]\label{def:combinFilling}
    Let $\mathcal{T}$ be a one-vertex triangulation of a compact orientable $3$-manifold with a single genus-$g$ boundary.
    Let $\mathcal{R}$ be a (possibly empty) set of edges in $\mathcal{T}$ corresponding to resolved attaching circles.
    Let $\mathcal{W}$ be the tuple of edge weights encoding the unresolved attaching circles.
    We call $(\mathcal{T},\mathcal{R},\mathcal{W})$ a \textit{combinatorial filling diagram}.
\end{definition}

Since $\mathcal{T}$ is assumed to be a one-vertex $3$-manifold triangulation in Definition~\ref{def:combinFilling}, we must have genus $g\geqslant1$
(otherwise the boundary of $\mathcal{T}$ would be a one-vertex triangulation of the $2$-sphere, which is impossible).

\begin{proposition}[Equivalence]\label{prop:diag-equiv-S=C}
    A topological filling diagram $(M, B,\gamma)$ can be converted to
    a combinatorial filling diagram $(\mathcal{T,R,W})$, and vice versa.
\end{proposition}
\begin{proof}
    \textit{Topological to combinatorial.}
    Let $\mathcal{T}$ denote a one-vertex triangulation of the $3$-manifold $M$;
    there are many ways to verify that such a triangulation exists, for example using tools discussed
    in~\cite[Section~5.2]{JacoRubinstein2006} or~\cite[Section~3]{Weeks2005}.
    After a small perturbation, we can ensure that the system $\gamma$ of attaching circles is disjoint from the vertex of
    $\mathcal{T}$, and transverse to the edges of $\partial \mathcal{T}$;
    we can then minimise edge weights in the usual way to ensure that $\gamma$ is normal with respect to $\partial \mathcal{T}$
    (see, for example, Matveev's book~\cite[Theorem~3.2.2]{Matveev2007} for a detailed account of this standard argument).
    We complete the conversion by using Proposition~\ref{prop:rooting} to turn the attaching circles into a filling bouquet,
    which we encode using edge weights $\mathcal{W}$ and resolved edges $\mathcal{R}$.

    \textit{Combinatorial to topological.}
	First, we can use the edge weights $\mathcal{W}$ to reconstruct the
	unresolved curves (uniquely, by Proposition~\ref{prop:weights=curves}).
	Provided that these unresolved curves, together with the resolved curves given by $\mathcal{R}$, form a valid set of rooted attaching circles,
	we can isotope all of these rooted circles away from each other to give a system $\gamma$ of attaching circles on $B=\partial \mathcal{T}$.
\end{proof}

\subsection{Updating weights after an edge-flip}\label{ssec:flipWeight}

A crucial feature of our main algorithm, which we introduce in the following section,
is that it flips edges of the boundary $2$-triangulation $B$ until every filling petal is resolved.
Each time we replace an edge $\textbf{\textit{e}}$ with a new edge $\textbf{\textit{e}}'$ via such an edge-flip, we must calculate
the correct weight to assign to $\textbf{\textit{e}}'$ so that the edge weights still describe the same combinatorial filling diagram.
Here we explain how this calculation is done.

Let $\textbf{\textit{e}}$ be an edge in a one-vertex triangulation $B$ of a closed orientable surface, and arbitrarily fix an orientation on $\textbf{\textit{e}}$.
Let $f_0$ and $f_1$ denote the two triangles of $B$ that are incident to $\textbf{\textit{e}}$ on either side.
For each $i\in\{0,1\}$, label the corners of $f_i$ as in Figure~\ref{fig:update-weights-labels}. That is:
\begin{itemize}
\item
Using the chosen orientation of $\textbf{\textit{e}}$, let $c_{i,0}$ and $c_{i,1}$ denote the corners at the start and end of $\textbf{\textit{e}}$, respectively.
\item
Let $c_{i,2}$ denote the corner opposite $\textbf{\textit{e}}$.
\end{itemize}
\begin{figure}[htbp]
    \centering
    \includegraphics[width=0.25\linewidth]{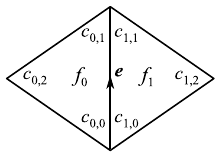}
    \caption{Local labeling of corners for the two triangles incident to an oriented edge.}
    \label{fig:update-weights-labels}
\end{figure}
For $i\in\{0,1\}$ and $j\in\{0,1,2\}$, let $n_{i,j}$ denote the normal arc coordinate that separates the corner $c_{i,j}$ from the rest of $f_i$,
and let $r_{i,j}$ denote the rooted arc coordinate in $f_i$ that intersects the edge opposite $c_{i,j}$.

In the quadrilateral given by the triangles $f_0$ and $f_1$, recall that an edge-flip replaces the diagonal edge $\textbf{\textit{e}}$ with the off-diagonal $\textbf{\textit{e}}'$;
note that prior to performing the edge-flip, $\textbf{\textit{e}}'$ is not an edge of the $2$-triangulation $B$.
Thus, calculating the weight on the new edge simply amounts to counting the number of times the combinatorial filling diagram intersects $\textbf{\textit{e}}'$.
Such points of intersection with $\textbf{\textit{e}}'$ come from three sources:
\begin{description}
\item[Normal arcs opposite $\textbf{\textit{e}}$.]
These give $N(\textbf{\textit{e}}) := n_{0,2} + n_{1,2}$ points of intersection;
see Figure~\ref{fig:update-weights-cases}, left.
\item[Root arcs disjoint from the interior of $\textbf{\textit{e}}$.]
These give $R(\textbf{\textit{e}}) := r_{0,0}+r_{0,1}+r_{1,0}+r_{1,1}$ points of intersection;
again, see Figure~\ref{fig:update-weights-cases}, left.
\item[Segments joining opposite edges of the quadrilateral $f_0\cup f_1$.]
Each such segment consists of a pair of normal arcs, either cutting off corners $c_{0,0}$ and $c_{1,1}$, or cutting off corners $c_{0,1}$ and $c_{1,0}$;
the case where $c_{0,0}$ and $c_{1,1}$ are cut off is illustrated in Figure~\ref{fig:update-weights-cases}, centre.
To count the number $X(\textbf{\textit{e}})$ of such segments, choose $i,j\in\{0,1\}$ such that
\[
n_{i,j} = \max\{ n_{0,0},\, n_{0,1},\, n_{1,0},\, n_{1,1} \}.
\]
Then $X(\textbf{\textit{e}}) = \max_{i,j}\{ 0,\, n_{i,j} - n_{1-i,j} - r_{1-i,2} \}$;
taking this maximum is necessary to account for cases where $n_{i,j} < n_{1-i,j} + r_{1-i,2}$,
such as the case illustrated in Figure~\ref{fig:update-weights-cases}, right.
\end{description}
Thus, after flipping $\textbf{\textit{e}}$, we have that the weight of the new edge $\textbf{\textit{e}}'$ is
\[ w(\textbf{\textit{e}}')=N(\textbf{\textit{e}})+R(\textbf{\textit{e}})+X(\textbf{\textit{e}}).\]
\begin{figure}[htbp]
    \centering
    \includegraphics[width=0.3\linewidth]{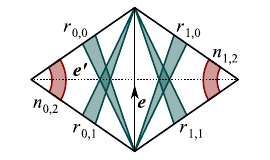}     \includegraphics[width=0.3\linewidth]{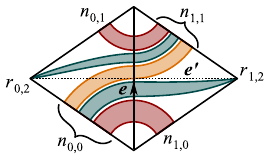}   \includegraphics[width=0.3\linewidth]{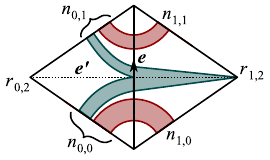}
    \caption{Three key cases that need to be accounted for when calculating the weight of the new edge after performing an edge flip.}
    \label{fig:update-weights-cases}
\end{figure}

\section{The combinatorial filling algorithm}\label{sec:main-algorithm}

The input for our algorithm is a combinatorial filling diagram.
Recall that this consists of the following three pieces of data:
\begin{itemize}
    \item a one-vertex triangulation of a compact orientable $3$-manifold with a single genus-$g$ boundary component;
    \item a set of boundary edges corresponding to the resolved filling petals; and
    \item a tuple of edge weights encoding the unresolved filling petals on the genus-$g$ boundary.
\end{itemize}

The time complexity of our algorithm, and the size of the triangulation that it constructs, depend on the following quantities:
the number $n$ of tetrahedra in the input $3$-triangulation,
the genus $g$ of the boundary surface,
and the total weight $\tau$ across all boundary edges.
For our analysis of the complexity, we will make the following assumptions:
\begin{itemize}
\item
We work in the word RAM model, so that we can perform all necessary arithmetic (such as in Section~\ref{ssec:flipWeight}) in constant time.
\item
We assume that in our data structure for triangulations, we can add tetrahedra and introduce new face-gluings in constant time.
This means, in particular, that layering and folding are constant-time operations.
\item
We assume that we can iterate through boundary edges of the triangulation in isolation (rather than iterating through all edges, and checking whether they are boundary).
This means that each such iteration only adds an $O(g)$ factor to our running time.
\end{itemize}
Under these assumptions, we will show that our algorithm runs in $O\left( g\tau+g^2 \right)$ time, and that it ultimately constructs a triangulation whose size is $O(n+g+\tau)$;
in particular, in the case where the input triangulation is a minimal layered handlebody, which has $3g-2$ tetrahedra, the size of the output triangulation is $O(g+\tau)$.

Although the running time of our algorithm scales linearly with $\tau$, the size of the edge weights is most reasonably measured by the number of bits needed to encode them.
Thus, our algorithm should really be considered to be exponential-time.

\begin{remark}
The initial preprint for this paper did not include an analysis of the running time of our algorithm.
We subsequently received correspondence from Ennes and Maria, who (as a small part of an otherwise unrelated project)
gave bounds on the running time of our algorithm~\cite[Theorem~5.6]{EnnesMaria2025HardnessQuantumArXivV3}.
The running time bounds given in this section are the same as theirs.
For completeness, we also include the (mostly routine) proofs for these bounds, which are omitted in~\cite{EnnesMaria2025HardnessQuantumArXivV3}.
\end{remark}

\subsection{Overview}
Let us now give an overview of our algorithm.
At a high-level, the algorithm consists of four major subroutines.
We give a rough summary of these subroutines and how they all fit together, before studying each individual subroutine in detail.

The first step of the algorithm is to layer tetrahedra onto the boundary, and hence
perform flips on the boundary 2-triangulation, until all filling petals are resolved.
We call the subroutine that achieves this the \textit{Petal-Resolver}.
At the end of this step all filling petals coincide with edges in the (new) boundary triangulation.

Next, we ensure each resolved filling petal is \textit{isolated}, meaning that the two triangles adjacent to this petal have no other resolved petals as edges.
If any of the resolved petals appear in the same triangle, then we adjust the triangulation by layering more tetrahedra;
an easy case is shown in Figure~\ref{fig:isolated-quads}.
We call this subroutine the \textit{Quad-Isolator}.
At the end of this step, the $g$ resolved filling petals appear as the diagonals of mutually disjoint quadrilaterals.

\begin{figure}[htbp]
    \centering
\begin{tikzpicture}[inner sep=0pt]
\node at (0,0) {
	\includegraphics[width=0.7\textwidth]{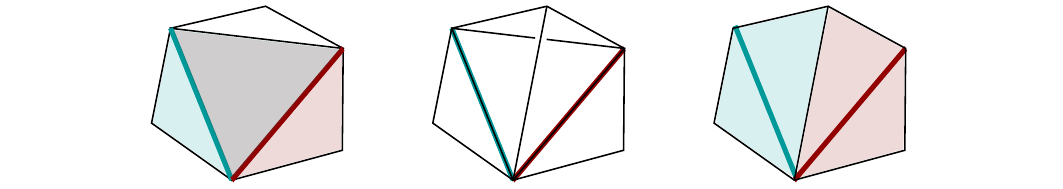}
};
\node at (-3.525,-0.1) {$b$};
\node at (-2.4,-0.3) {$r$};
\end{tikzpicture}
    \caption{The red and blue filling petals (labelled $r$ and $b$, respectively) share a triangle and are therefore \textit{not isolated} (left).
    The triangulation is adjusted by layering (centre) so that $r$ and $b$ coincide with edges forming diagonals of mutually disjoint quadrilaterals (right).}
    \label{fig:isolated-quads}
\end{figure}

The first topological change we make to the triangulation is done by \textit{folding} (recall Definition~\ref{def:folding}).
To prepare for folding, we layer an extra tetrahedron over each resolved filling petal.
The two exposed faces of each newly-layered tetrahedron are then identified by folding across the newly introduced edge.
This makes each of the filling petals homotopically trivial (recall Figure~\ref{fig:folding}), which we claim is analogous to attaching the wedges in the pinched filling algorithm.
Hence, we call this subroutine the \textit{Wedge-Folder}.
At the end of this step, the $3$-triangulation will be \emph{invalid}, in the sense that the link of the vertex is not a closed surface or a disc.

The remaining unidentified faces form a 2-triangulation of a (pinched) 2-sphere that we wish to fill trivially.
This is analogous to the step in the topological filling algorithm where we attach a 3-ball.
Hence, we refer to the final subroutine as the \textit{Ball-Filler}.
This routine may require further layering to adjust the boundary triangulation,
but ultimately the 2-sphere boundary is closed by again folding across edges to identify adjacent faces.
Note that the \textit{Ball-Filler} may fold two faces together regardless of whether those two faces belong to the same tetrahedron;
this is in contrast to the \textit{Wedge-Folder}, which always performs a layering before each fold, and therefore always folds together two faces of a single tetrahedron.

Once all this is done, the algorithm returns a one-vertex triangulation of a closed 3-manifold.
For the remainder of this section, we consider each subroutine in more detail.

\subsection{Petal-Resolver}\label{ssec:PR}
The Petal-Resolver layers tetrahedra onto the genus-$g$ boundary to adjust the triangulation in such a way that each filling petal becomes resolved
(that is, each petal appears as a boundary edge of the triangulation).
Recall that the total weight of the boundary edges counts the number of points in which \emph{unresolved} petals intersect
boundary edges of the triangulation, so the goal is to reduce this total weight to zero.
We say that a boundary edge $\textbf{\textit{e}}$ is \textit{reducible} if performing the flip to replace $\textbf{\textit{e}}$ would strictly decrease the total weight, and \textit{irreducible} otherwise.

\begin{algorithm}[Petal-Resolver]\label{algm:petal-resolver}
Given a one-vertex 3-manifold triangulation $\mathcal{T}$ with genus-$g$ boundary, together with a valid collection of filling petals with total weight $\tau$:
	\begin{enumerate}
	\item\label{step:PR-find-edge}
	If all filling petals are resolved, then terminate and return the triangulation $\mathcal{T}$.
	Otherwise, use the following loop to find a reducible boundary edge in $\mathcal{T}$:
		\begin{enumerate}[(a)]
		\item
		Let $\textbf{\textit{e}}$ denote the next boundary edge of $\mathcal{T}$, and let $W$ be the weight of $\textbf{\textit{e}}$.
		\item
		Follow Section~\ref{ssec:flipWeight} to calculate the weight $W'$ of the new boundary edge that would result from flipping $\textbf{\textit{e}}$.
		\item
		If $W'<W$, then go to Step~\ref{step:PR-layer} with this weight $W'$ and this reducible edge $\textbf{\textit{e}}$.
		Otherwise, go back to the top of this loop.
		\end{enumerate}
\item\label{step:PR-layer}
Modify $\mathcal{T}$ by layering a tetrahedron over $\textbf{\textit{e}}$, thereby modifying $\partial \mathcal{T}$ via a flip replacing $\textbf{\textit{e}}$ with a new boundary edge $\textbf{\textit{e}}'$.
Assign the weight $W'$ to $\textbf{\textit{e}}'$, and go back to Step~\ref{step:PR-find-edge} with the now-modified triangulation $\mathcal{T}$.
	\end{enumerate}
\end{algorithm}

\begin{proposition}\label{prop:PR-terminates}
    If one or more filling petals remains unresolved, then there must be at least one reducible edge.
    Hence, the Petal-Resolver is guaranteed to terminate in $O(g\tau)$ time.
    Moreover, the Petal-Resolver adds at most $\tau$ new tetrahedra (via layering).
\end{proposition}
\begin{proof}
The hard part is showing that it is always possible to find a reducible edge, so we first deal with the claims about the complexity of the Petal-Resolver.
Observe that Step~\ref{step:PR-find-edge} iterates through $O(g)$ boundary edges but otherwise only needs to perform constant-time arithmetic.
Step~\ref{step:PR-layer} adds a single tetrahedron via layering (in constant time), and is guaranteed to reduce the total weight.
Therefore, the Petal-Resolver repeats these two steps at most $\tau$ times, so the overall running time is $O(g\tau)$ and the number of new tetrahedra is at most $\tau$.

We devote the rest of this proof to showing that there is always at least one reducible edge.
In fact, we will show that at least one of the \textit{maximal-weight} edges must be reducible.

Consider two adjacent triangles where the shared edge \textbf{\textit{e}} has maximal weight amongst all edges in the boundary triangulation.
The two triangles must be distinct, since a triangle cannot be adjacent to itself in the boundary triangulation
(this would require either an extra vertex or that the boundary were non-orientable).

For the quadrilateral formed by these two triangles, we label vertices and edges as in Figure~\ref{fig:PR-quad-labels}.
We know that the four vertices are actually copies of the same vertex, and the edges may also be identified in pairs;
however, we assign them distinct labels for clarity.
We label the types of arcs that appear in a quadrilateral as per Figure~\ref{fig:PR-quad-labels}.
Let $v_i$ be the number of arcs in the quadrilateral cutting off vertex $i\in\{1,2,3,4\}$.
Let $x$ be the number of arcs that cut across the quadrilateral, either through edges $\textbf{\textit{a}},\,\textbf{\textit{e}},\,\textbf{\textit{b}}$ or edges $\textbf{\textit{d}},\,\textbf{\textit{e}},\,\textbf{\textit{c}}$.
Let $r_{\textbf{\textit{a}}}$ and $r_{\textbf{\textit{d}}}$ be the number of arcs that are
rooted at vertex 3, and that leave the quadrilateral via edge $\textbf{\textit{a}}$ and $\textbf{\textit{d}}$, respectively.
Define $r_{\textbf{\textit{b}}}$ and $r_{\textbf{\textit{c}}}$ similarly, but for arcs rooted at vertex 4.
Since we assume \textbf{\textit{e}} is maximal-weight, Corollary~\ref{cor:max-weight-root} implies that there can be no arcs rooted at vertices 1 or 2.

\begin{figure}[htbp]
    \centering
    \includegraphics[width=0.28\textwidth]{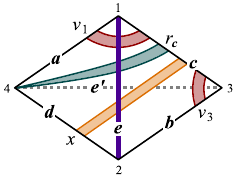}
    \caption{Labels for edges and arcs in a quadrilateral that are used in the proof of Proposition~\ref{prop:PR-terminates}.}
    \label{fig:PR-quad-labels}
\end{figure}

\paragraph*{Case 1:} Both triangles contain rooted arcs.
\begin{enumerate}[label={(\roman*)}]
\item
Suppose a filling petal $\gamma$ appears as an off-diagonal of the quadrilateral, as in Figure~\ref{subfig:PR-both-roots-a} or Figure~\ref{subfig:PR-both-roots-b}.
Performing the flip in either of these quadrilaterals exchanges the edge \textbf{\textit{e}} for an edge that coincides with the filling petal $\gamma$.
    In particular, \textbf{\textit{e}} is reducible.
    \item If there is no filling petal appearing as an off-diagonal of the quadrilateral,
    we have the scenario shown in Figure~\ref{subfig:PR-both-roots-c} (up to symmetry).
    Suppose there are $x\geqslant 0$ normal arcs that cross the quadrilateral between the rooted arcs.
    Further suppose that there are $v_1\geqslant 0$ normal arcs cutting off vertex 1, and $v_2\geqslant 0$ cutting off vertex 2.
    Finally, suppose there are $r_\textbf{\textit{c}}\geqslant 1$ arcs rooted at vertex 4 and $r_\textbf{\textit{d}}\geqslant 1$ rooted arcs rooted at vertex 3.
    In this case, performing the flip strictly reduces the edge weight, since    $w(\textbf{\textit{e}})=v_1+r_\textbf{\textit{c}}+x+r_\textbf{\textit{d}}+v_2$,
    while the weight of the flipped edge would be $w(\textbf{\textit{e}}')=x$.
    Therefore \textbf{\textit{e}} is reducible.
\end{enumerate}
\begin{figure}[htbp]
        \centering
        \begin{subfigure}{0.28\textwidth}
            \centering
            \includegraphics[width=\linewidth]{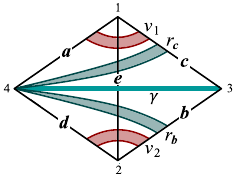}
            \caption{Case 1(i).}
            \label{subfig:PR-both-roots-a}
        \end{subfigure}
        \hfill
        \begin{subfigure}{0.28\textwidth}
            \centering
            \includegraphics[width=\linewidth]{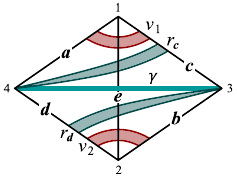}
            \caption{Case 1(i).}
            \label{subfig:PR-both-roots-b}
        \end{subfigure}
        \hfill
        \begin{subfigure}{0.28\textwidth}
            \centering
            \includegraphics[width=\linewidth]{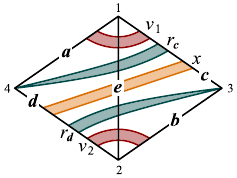}
            \caption{Case 1(ii).}
            \label{subfig:PR-both-roots-c}
        \end{subfigure}
        \caption{All possible configurations of arcs when both of the triangles contain rooted arcs, as in Case~1.}
        \label{fig:PR-both-roots}
    \end{figure}

\paragraph*{Case 2:} Only one triangle contains rooted arcs.
\begin{enumerate}[label={(\roman*)}]
    \item If there are multiple rooted arcs that leave the quadrilateral through each of the two opposite edges,
    we have the scenario shown in Figure~\ref{subfig:PR-one-root-a} (up to symmetry).
    In this case, $w(\textbf{\textit{e}})=v_1+r_\textbf{\textit{c}}+r_\textbf{\textit{b}}+v_2$,
    with $v_1,v_2\geqslant 0$ and $r_\textbf{\textit{c}},r_\textbf{\textit{b}}\geqslant 1$.
    Meanwhile, the weight of the flipped edge is $v_3\geqslant 0$.
    Note that $w(\textbf{\textit{c}})=v_1+r_\textbf{\textit{c}}+v_3$ can be at most equal to the weight of \textbf{\textit{e}},
    by the assumption that \textbf{\textit{e}} is maximal, so $v_1+r_\textbf{\textit{c}}+r_\textbf{\textit{b}}+v_2\geqslant v_1+r_\textbf{\textit{c}}+v_3 >v_3$.
    Since the flipped edge would have weight $w(\textbf{\textit{e}}')=v_3$,
    performing the flip strictly reduces the total weight, and hence \textbf{\textit{e}} is reducible.
    \item If all rooted arcs leave the quadrilateral through the same opposite edge, then
    we have the scenario shown in Figure~\ref{subfig:PR-one-root-b} (up to symmetry).
    Similar to the previous case, we have $w(\textbf{\textit{e}})=v_1+r_\textbf{\textit{c}}+x+v_2$, which (by the maximality assumption)
    is at least as big as the weight of edge $\textbf{\textit{c}}$, $w(\textbf{\textit{c}})=v_1+r_\textbf{\textit{c}}+x+v_3$.
    Together with the fact that $r_\textbf{\textit{c}}\geqslant 1$, this gives us that
    $v_1+r_\textbf{\textit{c}}+x+v_2\geqslant v_1+r_\textbf{\textit{c}}+x+v_3 > x+v_3$.
    Since the flipped edge would have weight $w(\textbf{\textit{e}}')=x+v_3$
    we again see that performing the flip strictly reduces the total weight, and hence that \textbf{\textit{e}} is reducible.
\end{enumerate}

\begin{figure}[htbp]
        \centering
        \begin{subfigure}{0.28\textwidth}
            \centering
            \includegraphics[width=\linewidth]{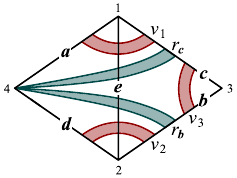}
            \caption{Case 2(i).}
            \label{subfig:PR-one-root-a}
        \end{subfigure}
        $\qquad$
        \begin{subfigure}{0.28\textwidth}
            \centering
            \includegraphics[width=\linewidth]{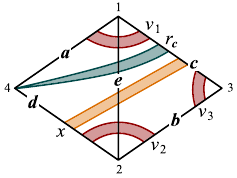}
            \caption{Case 2(ii).}
            \label{subfig:PR-one-root-b}
        \end{subfigure}
        \hfill
        \caption{All possible configurations of arcs when exactly one of the two triangles contains rooted arcs, as in Case~2. }
        \label{fig:PR-one-root}
    \end{figure}

\paragraph*{Case 3:} Neither triangle contains a rooted arc.
If \textbf{\textit{e}} is reducible then we are done, so assume it is irreducible.
We will show that there must be some other (maximal-weight) edge that is reducible.

To this end, first notice that there can never be arcs that cut the quadrilateral in half in both directions at the same time, since such arcs would intersect.
Without loss of generality, we assume $x$ is the number of arcs cutting through edges $\mathbf{d,e,c}$, as in Figure~\ref{fig:PR-neither-root}.

\begin{figure}[htbp]
    \centering
    \includegraphics[width=0.28\textwidth]{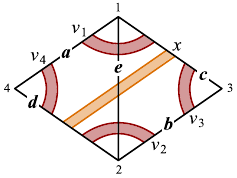}
    \caption{Configuration of arcs in a quadrilateral where neither triangle contains a root, as in Case~3.}
    \label{fig:PR-neither-root}
\end{figure}

Under the maximality assumption, we have that $w(\textbf{\textit{e}})=v_1+x+v_2\geqslant v_1+x+v_3=w(\textbf{\textit{c}})$
and $w(\textbf{\textit{e}})=v_1+x+v_2\geqslant v_4+x+v_2=w(\textbf{\textit{d}})$.
These respectively imply that $v_2\geqslant v_3$ and $v_1\geqslant v_4$.

Note that performing the flip will result in an edge of weight $w(\textbf{\textit{e}}')=v_3+x+v_4$.
Under the assumption that \textbf{\textit{e}} is irreducible, $w(\textbf{\textit{e}}')\geqslant w(\textbf{\textit{e}})$, which implies that $v_3+v_4\geqslant v_1+v_2$.
In summary, $v_3+v_4\geqslant v_1+v_2$ by the irreducibility assumption, and the maximality assumption
implies both $v_2\geqslant v_3$ and $v_1\geqslant v_4$, so altogether we have $v_2=v_3$ and $v_1=v_4$.

This means that edges $\textbf{\textit{c}}$ and $\textbf{\textit{d}}$, which have weights given by
$w(\textbf{\textit{c}})=v_1+x+v_3$ and $w(\textbf{\textit{d}})=v_2+x+v_4$, are also maximal-weight edges.
Note that one or both of $\textbf{\textit{a}}$ and $\textbf{\textit{b}}$ may also be maximal-weight edges.
If any of the additional maximal-weight edges are reducible, then we are done.

Hence, let us consider the case where all maximal-weight edges are also irreducible;
by our previous arguments, this means that none of these maximal-weight edges are incident to triangles containing rooted arcs.
Beginning with the quadrilateral just discussed, inductively build an \textit{irreducible region} by
including the triangles adjacent to each maximal-weight irreducible edge (if they are not already included).
That is, for each maximal-weight edge \textbf{\textit{f}} along which we include an additional
triangle, we can make the same arguments for the quadrilateral at \textbf{\textit{f}};
this yields additional maximal-weight edges, along which we might need to include further triangles.
This process must terminate, since the total triangulation is finite. The result is a subcomplex in which all interior edges are maximal-weight and irreducible.

Using the labels of Figure~\ref{fig:PR-neither-root}, and assuming the quadrilateral is contained inside the irreducible region,
notice that $\textbf{\textit{c}}$, $\textbf{\textit{d}}$ and $\textbf{\textit{e}}$ must each be maximal-weight and irreducible,
and that each of $\textbf{\textit{a}}$ and $\textbf{\textit{b}}$ may also be maximal-weight and irreducible.
In particular, notice that each triangle in the irreducible region must be incident to at least two maximal-weight irreducible edges.
We call a triangle \emph{equilateral} if all three of its edges are maximal-weight and irreducible,
and we call it \emph{isosceles} if exactly two of its edges are maximal-weight and irreducible.
With this terminology, we can say that every triangle in the irreducible region is either isosceles or equilateral.

From here, we argue that there must be at least one curve contained entirely in the irreducible region.
Such a curve must consist entirely of normal arcs, and is therefore not rooted.
This contradicts the assumption that we started with a valid collection of filling petals.

To find such an unrooted curve, let $m$ denote the maximal weight, and consider an arbitrary weight-$m$ edge $\textbf{\textit{e}}$ in the irreducible region.
Let $p_1,\ldots,p_m$ denote the points at which the filling petals intersect $\textbf{\textit{e}}$, labelled in the order that they appear along $\textbf{\textit{e}}$.
If $m$ is odd, then we write $m=2k+1$ and call the point $p_{k+1}$ the \emph{innermost intersection};
otherwise, if $m$ is even, then we write $m=2k$ and call the points $p_k$ and $p_{k+1}$ the \emph{innermost intersections}.

With this terminology in mind, consider an arbitrary triangle $\Delta$ in the irreducible region.
If $\Delta$ is equilateral, then observe that $m$ must be even;
moreover, the innermost intersections are all paired together by three \emph{innermost arcs}, as shown in Figure~\ref{fig:innermost-intersections} (left).
On the other hand, if $\Delta$ is isosceles, then the innermost intersections are paired together by
either one or two (depending on the parity of $m$) \emph{innermost arcs}, as shown in Figure~\ref{fig:innermost-intersections} (right and centre).

\begin{figure}[htbp]
    \centering
    \includegraphics[width=0.6\textwidth]{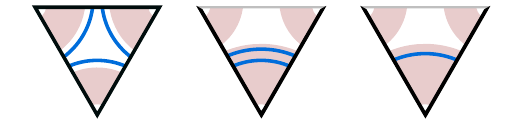}
    \caption{The $3$, $2$ or $1$ (respectively, from left to right) innermost arcs that are present in an arbitrary triangle $\Delta$ in an irreducible region.
    Left: When $\Delta$ is equilateral and $m$ is even.
    Centre: When $\Delta$ is isosceles and $m$ is even.
    Right: When $\Delta$ is isosceles and $m$ is odd.}
    \label{fig:innermost-intersections}
\end{figure}

Notice that every innermost arc has both endpoints on a maximal-weight irreducible edge;
in other words, every innermost arc must be disjoint from the boundary of the irreducible region.
Thus, the innermost arcs must all join together to give at least one closed curve that never leaves the irreducible region, which produces the desired contradiction.
Hence, we conclude that we are guaranteed to find a reducible maximal-weight edge.

Since there is always a reducible edge, the Petal-Resolver is able to strictly reduce
the total edge weight until it reaches 0, at which point all filling petals are resolved.
\end{proof}

\subsection{Quad-Isolator}
Our main algorithm makes the filling petals homotopically trivial by folding them back on themselves.
To be able to perform such folds along all the petals simultaneously, we need the petals to form diagonals of disjoint quadrilaterals.
That is, if for each (resolved) filling petal we consider the quadrilateral formed by the two boundary triangles incident to the petal,
then we require the interiors of these quadrilaterals to be mutually disjoint.

The Quad-Isolator subroutine uses layering to ensure that this requirement is satisfied.
The order in which the layerings are performed is important, and is determined by a certain labelling of the boundary edges.
We first state how the Quad-Isolator subroutine constructs the labelling, and how it uses this labelling to perform the required layerings;
we then explain why this subroutine works as intended.

\begin{algorithm}[Quad-Isolator]\label{algm:quad-isolator}
    Given a one-vertex $3$-manifold triangulation with genus-$g$ boundary $B$,
    along with $g$ distinguished edges of $B$ corresponding to resolved filling petals:
\begin{enumerate}
    \item\label{step:QI-label0}
    Label all edges corresponding to resolved filling petals with a 0.
    \item\label{step:QI-labelling}
    Starting with $i=1$, repeat the following until there are no triangles in $B$ with exactly two labelled edges:
    	\begin{enumerate}
    	\item Identify all triangles for which exactly two edges have labels
    	(each of these labels will be less than $i$).
    	\item For each triangle identified above, assign the third edge the label $i$.
    	\item Increase $i$ by one.
    	\end{enumerate}
    \item\label{step:QI-layering}
    Starting with $i$ equal to the largest edge label in the triangulation, repeat the following:
    	\begin{enumerate}
    	\item\label{sstep:QI-layering}
    	Layer a tetrahedron over each $i$-labelled edge.
    	\item Decrease $i$ by one.
    	Terminate if $i=0$.
    	\end{enumerate}
\end{enumerate}
\end{algorithm}

In the forthcoming analysis of Algorithm~\ref{algm:quad-isolator}, we refer to an edge labelled $i$ as an $i$-edge.
We call any triangle with exactly two $0$-edges a \textit{seed}.
For $i>0$, define a \textit{flip region of depth $i$} to be a collection of triangles bounded by some number of $0$-edges and a single $i$-edge,
where all edges interior to this region are assigned labels strictly between $0$ and $i$;
define a flip region of depth $0$ to be an isolated $0$-edge.
For example, the two possible flip regions of depth 2 (up to symmetry) are shown in Figure~\ref{fig:QI-flip-region-2}.

    \begin{figure}[htbp]
        \centering
        \includegraphics[width=0.75\textwidth]{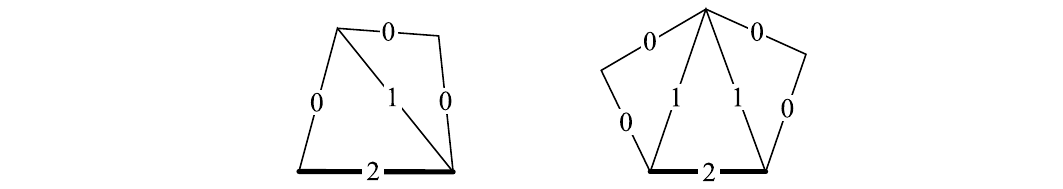}
        \caption{The two flip regions of depth 2 (up to symmetry).}
        \label{fig:QI-flip-region-2}
    \end{figure}

\begin{lemma}\label{lem:QI-flip-regions}
If an edge is labelled $i$, then it must belong to exactly one flip region of depth $i$.
\end{lemma}
\begin{proof}
The lemma holds trivially for $i=0$, so assume for the rest of this proof that $i>0$.

We first show that every $i$-edge belongs to at least one flip region of depth $i$.
This holds for $i=1$ because edges with label $1$ only ever appear as the third edge in a seed;
a seed is exactly a flip region of depth $1$.

    Next consider $i=2$.
    An edge with label $2$ belongs to a boundary triangle $\Delta$ for which the other edges are labelled either $0$ and $1$, or $1$ and $1$.
    In either case, the edges with label $1$ must belong to a seed on the other side of $\Delta$.
    So altogether, both scenarios imply that the $2$-edge is
    the only non-zero labelled boundary edge of a region otherwise bounded by $0$-edges, whose interior edges are all labelled $1$.
    Recall the possibilities from Figure~\ref{fig:QI-flip-region-2}.

    Now consider an arbitrary integer $k>2$.
    An edge will only ever be labelled with a $k$ if it belongs to a boundary triangle $\Delta$ in which
    one of the other edges is labelled $k-1$, and the third edge is labelled with some integer from $0$ to $k-1$.
    Assuming that the lemma is true for all $i<k$, we see that $\Delta$ must be incident to two flip regions of depth less than $k$.
    The union of $\Delta$ with these two flip regions forms a flip region of depth $k$.
By induction, this shows that every $k$-edge belongs to at least one flip region of depth $k$.

To complete the proof, let \textbf{\textit{e}} be an $i$-edge, for any label $i>0$.
We have just shown that there is a flip region of depth $i$ on at least one side of \textbf{\textit{e}},
but we still need to rule out the possibility that such a region appears on both sides of \textbf{\textit{e}}.
For this, observe that the union of two distinct regions of depth $i$ incident to \textbf{\textit{e}} would form a  triangulated polygon $P$ bounded entirely by $0$-edges.
To see why such a polygon $P$ cannot exist, note that the boundary triangulation $B$ (which is a one-vertex triangulation of a genus-$g$ surface) only has
$g$ edges with $0$-labels, and that each such edge can appear at most twice in the boundary of $P$.
Thus, $P$ is a polygon with no interior vertices and at most $2g$ boundary edges, which implies that $P$ can contain at most $2g-2$ triangles.
But there are $4g-2$ triangles in $B$, so the set of $0$-edges bounding $P$ must form a separating multicurve in $B$.
This contradicts the fact that the $0$-edges correspond to a valid collection of filling petals.
\end{proof}

\begin{proposition}\label{prop:QI-terminates}
    The Quad-Isolator terminates in $O(g^2)$ time, and is guaranteed to produce a triangulation in which
    all the filling petals form diagonals of mutually disjoint quadrilaterals.
    Moreover, the Quad-Isolator adds $O(g)$ new tetrahedra (via layering).
\end{proposition}
\begin{proof}
For the time complexity, Step~\ref{step:QI-label0} requires $O(g)$ time to assign the label $0$ to the $g$ filling petals.
Then, each iteration of the loop in Step~\ref{step:QI-labelling} requires checking $O(g)$ boundary triangles;
at worst, this loop repeats $O(g)$ times, and so Step~\ref{step:QI-labelling} requires $O(g^2)$ time overall.
Finally, Step~\ref{step:QI-layering} layers over each boundary edge at most once, so this adds $O(g)$ new tetrahedra in $O(g)$ time.
Altogether, the Quad-Isolator runs in $O(g^2)$ time and adds $O(g)$ new tetrahedra, as claimed.

To understand why the Quad-Isolator is guaranteed to produce the desired result (all filling petals forming diagonals of mutually disjoint quadrilaterals),
we first show that the following properties remain true throughout the execution of Step~\ref{step:QI-layering}:
\begin{enumerate}[(I)]
\item\label{inv:twoEdges}
No triangle in $B$ has exactly two labelled edges.
\item\label{inv:disjointQuads}
All the edges of $B$ with maximum valued $i$-labels must always form diagonals of mutually disjoint quadrilaterals.
\end{enumerate}
Property~\ref{inv:disjointQuads} is especially important because it ensures that Step~\ref{sstep:QI-layering} is well-defined:
there is one and only one way to layer over all edges with maximum valued $i$-labels.
To establish these two properties, our strategy is to inductively show that both properties remain true after every layering performed in Step~\ref{step:QI-layering}.

Before we give the inductive argument, we first show that property~\ref{inv:disjointQuads} follows from property~\ref{inv:twoEdges}.
To do this, let $i$ denote the maximum value of all edge labels, and suppose property~\ref{inv:twoEdges} holds.
If property~\ref{inv:disjointQuads} fails, then the $2$-triangulation $B$ must have a face $\Delta$ that is incident to more than one $i$-edge.
By property~\ref{inv:twoEdges}, all three edges of $\Delta$ must be labelled;
that is, two edges of $\Delta$ must be $i$-edges, and the third edge must be a $j$-edge for some $j\leqslant i$.
By Lemma~\ref{lem:QI-flip-regions}, $\Delta$ must therefore be incident to two flip regions
of depth $i$ and one flip region of depth $j$, as shown in Figure~\ref{subfig:QI-contradiction}.
Taking the union of $\Delta$ with these three flip regions gives a triangulated polygon $P$ whose boundary consists entirely of $0$-edges.
Following the same argument as at the end of the proof of Lemma~\ref{lem:QI-flip-regions}, such a polygon $P$ cannot exist.

Turning now to the inductive argument, observe that immediately after Step~\ref{step:QI-labelling} terminates, property~\ref{inv:twoEdges} holds by construction.
Therefore property~\ref{inv:disjointQuads} also holds.
This establishes the base case.

For the inductive step, assume that both properties hold, and as before let $i$ denote the maximum value of all edge labels.
We claim that property~\ref{inv:twoEdges} remains true after layering over an $i$-edge \textbf{\textit{e}}.
To see why, let $\Delta$ and $\Delta'$ denote the two boundary triangles on either side of \textbf{\textit{e}}.
By Lemma~\ref{lem:QI-flip-regions}, \textbf{\textit{e}} only belongs to one flip region of depth $i$;
assume without loss of generality that $\Delta$ lies inside this flip region.
Under this assumption, we know that the following is true prior to layering over \textbf{\textit{e}} (see Figure~\ref{subfig:QI-generic}):
\begin{itemize}
\item
Apart from the $i$-edge \textbf{\textit{e}}, $\Delta$ must also have an $(i-1)$-edge and an $(i-j)$-edge, for some $0< j\leqslant i$.
\item
$\Delta'$ does not belong to any flip region, and hence the two edges of $\Delta'$ other than \textbf{\textit{e}} must both be unlabelled.
\end{itemize}
We deduce that layering over \textbf{\textit{e}} produces the configuration of boundary triangles shown in Figure~\ref{subfig:QI-splitting}.
In particular, this layering creates two new boundary triangles, each of which has exactly one labelled edge.
Thus, property~\ref{inv:twoEdges} remains true after this layering, and hence property~\ref{inv:disjointQuads} also remains true.
By induction, we therefore conclude that both properties remain true throughout the execution of Step~\ref{step:QI-layering}.

    \begin{figure}[htbp]
    \begin{subfigure}{0.32\textwidth}
        \centering
        \includegraphics[width=0.8\textwidth]{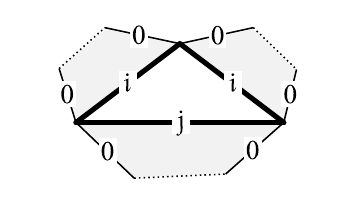}
        \caption{If $\Delta$ is incident to two $i$-edges, the third edge is a $j$-edge for some $j\leq i$.
        }
        \label{subfig:QI-contradiction}
    \end{subfigure}
    \begin{subfigure}{0.32\textwidth}
        \centering
        \includegraphics[width=0.8\textwidth]{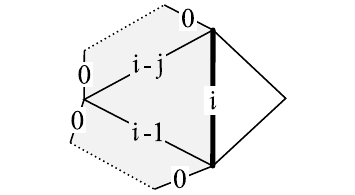}
        \caption{The generic configuration for a maximum valued $i$-edge.}
        \label{subfig:QI-generic}
    \end{subfigure}
    \begin{subfigure}{0.32\textwidth}
        \centering
        \includegraphics[width=0.8\textwidth]{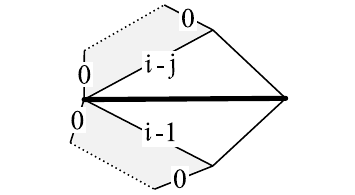}
        \caption{Layering over the $i$-edge leaves flip regions of depths $i-j$ and $i-1$.}
        \label{subfig:QI-splitting}
    \end{subfigure}
    \caption{Different configurations of flip regions in the proof of Proposition~\ref{prop:QI-terminates}.}
    \label{fig:QI-prop-proof}
    \end{figure}

To complete the proof, we continue with the same notation as above:
let \textbf{\textit{e}} be an edge with maximum valued label $i$.
As before, \textbf{\textit{e}} belongs to a flip region $R$ of depth $i$, and we have the configuration of boundary triangles shown in Figure~\ref{subfig:QI-generic}.
After layering over \textbf{\textit{e}}, we end up with the configuration shown in Figure~\ref{subfig:QI-splitting}.
Notice that the flip region $R$ gets split into two smaller flip regions with respective depths $i-1$ and $i-j$, where $0< j\leqslant i$.
Crucially, both new flip regions have depth strictly less than $i$.

The upshot is that each iteration of Step~\ref{step:QI-layering} reduces
the maximum depth of all flip regions until the only remaining labelled edges are the $0$-edges.
By property~\ref{inv:disjointQuads}, we conclude that Step~\ref{step:QI-layering} terminates with the desired result:
the $0$-edges all form diagonals of mutually disjoint quadrilaterals.
\end{proof}

\subsection{Wedge-Folder}
The following subroutine layers and folds tetrahedra that correspond to the wedges of the pinched filling algorithm, thus making the filling petals homotopically trivial.

\begin{algorithm}[Wedge-Folder]\label{algm:wedge-folder}
    Given a one-vertex $3$-manifold triangulation with genus-$g$ boundary, along with
    $g$ resolved filling petals that form diagonals of mutually disjoint quadrilaterals:
    \begin{enumerate}
        \item\label{step:WF-layer}
        \textit{Layer over each filling petal.}
        Before we can fold each filling petal back on itself, it must appear as an \textit{off-diagonal}.
        We achieve this by layering a tetrahedron over each filling petal.
        \item\label{step:WF-fold}
        \textit{Fold along the filling petals.}
        We identify the two boundary faces of each newly layered tetrahedron by
        folding across the edge that was introduced by the layering.
    \end{enumerate}
\end{algorithm}

This subroutine involves layering on $g$ tetrahedra and performing $g$ folds, so it terminates, and leaves $2g-2$ remaining boundary faces.

\begin{example}\label{eg:WF-genus-3}
    Consider the boundary triangulation for a genus-$3$ handlebody shown in Figure~\ref{fig:WF-genus-3}.
    We see (schematically) what happens when we perform the folds to make the filling petals homotopically trivial.
    Notice that the folds identify edges that were not previously in the same edge class,
    and the remaining boundary faces form a four-vertex triangulation $B$ of $S^2$.
Since the $3$-triangulation $\mathcal{T}$ remains one-vertex, the boundary $2$-triangulation $B$ cannot be embedded in $\mathcal{T}$;
instead, the four vertices of $B$ are all pinched together (as defined in Section~\ref{ssec:pinched-manifolds}) at the single vertex of $\mathcal{T}$.

\begin{figure}[htbp]
    \centering
    \includegraphics[width=0.75\textwidth]{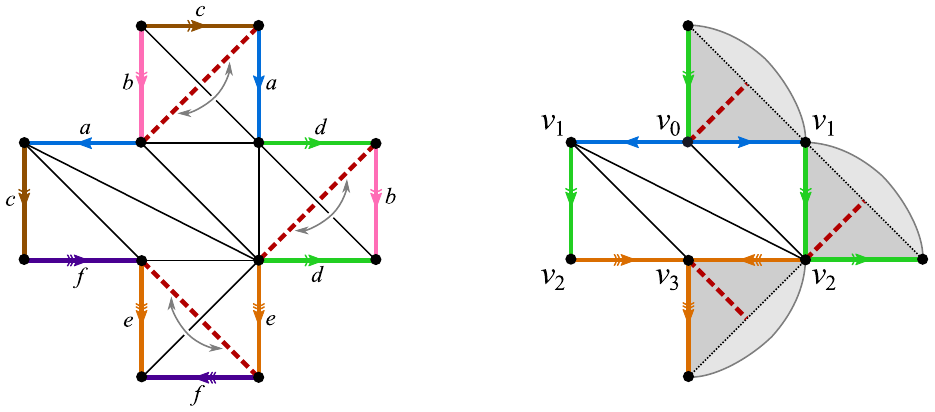}
    \caption{(Example~\ref{eg:WF-genus-3}). Left: A one-vertex genus-$3$ boundary $2$-triangulation, with filling petals indicated by dashed red lines.
    Right: The result of performing the folds that make the filling petals trivial; edges $b$, $c$, $d$ are all glued together, as are edges $e$ and $f$. Note that after these folds, the boundary $2$-triangulation becomes a sphere whose four vertices $v_0$, $v_1$, $v_2$ and $v_3$ are all pinched together.
    }
    \label{fig:WF-genus-3}
\end{figure}
\end{example}

\begin{proposition}\label{prop:WF-is-legit}
The Wedge-Folder routine terminates in $O(g)$ time, and adds exactly $g$ new tetrahedra (via layering).
Moreover, upon termination, the following must hold:
\begin{enumerate}[(a)]
\item\label{sprop:WF-link}
The link of the vertex of the $3$-triangulation is a genus-$0$ surface.
\item\label{sprop:WF-bdry}
The remaining boundary faces of the $3$-triangulation form a pinched $2$-sphere.
\end{enumerate}
\end{proposition}
\begin{proof}
For the complexity, Step~\ref{step:WF-layer} layers on exactly $g$ new tetrahedra in $O(g)$ time.
Step~\ref{step:WF-fold} performs exactly $g$ folds, and therefore also requires $O(g)$ time.

Let $L$ denote the link of the vertex of the $3$-triangulation;
note that $L$ begins as a disc, since we start with a (valid) triangulation of a $3$-manifold with boundary.
After layering over all the filling petals, $L$ remains a disc.
The boundary circle of $L$ is partitioned into $12g-6$ normal arcs (given by the intersection of this circle with the boundary triangles of the $3$-triangulation).
Each time we fold along a filling petal, we identify three pairs of these normal arcs (see Figure~\ref{fig:link-fold}), thereby changing the surface $L$.
We will show that after all the folds have been performed, $L$ remains a genus-$0$ surface.

\begin{figure}
    \centering
    \includegraphics[width=0.8\linewidth]{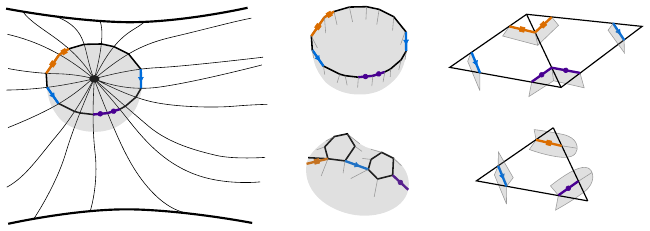}
    \caption{Folding a pair of adjacent boundary faces together changes the vertex link by identifying three pairs of arcs.
    Left: An example of how the three pairs of arcs (coloured and decorated) might appear on the boundary of the vertex link (shaded).
    Centre: The vertex link before (top) and after (bottom) identifying the three pairs of arcs.
    Top right: An example of a pair of adjacent boundary faces, such that folding these faces together would
    cause the three pairs of coloured/decorated arcs to be identified.
    Bottom right: The internal face that we obtain after performing the fold.}
    \label{fig:link-fold}
\end{figure}

To this end, we first consider the possible ways in which identifying just one pair of normal arcs can change $L$.
We have the following four cases (illustrated in Figure~\ref{fig:link-fold-cases}):
\begin{enumerate}[(i)]
\item\label{link-pinch}
If the two normal arcs intersect at precisely one endpoint, then identifying these arcs preserves both the genus and the number of boundary circles of $L$.
\item\label{link-close}
If the two normal arcs intersect at both endpoints, then identifying these arcs preserves the genus of $L$ but removes the boundary circle given by the union of the two arcs.
\item\label{link-across}
If the two normal arcs are disjoint but belong to the same boundary circle of $L$, then identifying these arcs preserves the genus of $L$ but creates an extra boundary circle.
\item\label{link-genus}
If the two normal arcs belong to distinct boundary circles of $L$, then identifying these arcs adds genus to $L$.
\end{enumerate}
Thus, it suffices to show that the folds never identify two normal arcs in different boundary circles of $L$.
For each fold, such an identification can only arise from one of the three pairs of normal arcs:
the pair that intersects the off-diagonal along which we perform the fold.

\begin{figure}
    \centering
    \begin{subfigure}[t]{0.24\textwidth}
    \centering
    	\includegraphics[width=\linewidth]{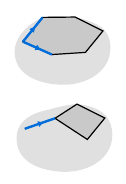}
    \caption{Case~\ref{link-pinch}.}
    \label{sfig:link-pinch}
    \end{subfigure}
    \hfill
    \begin{subfigure}[t]{0.24\textwidth}
    \centering
    	\includegraphics[width=\linewidth]{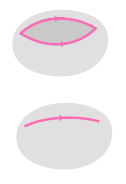}
    \caption{Case~\ref{link-close}.}
    \label{sfig:link-close}
    \end{subfigure}
    \hfill
    \begin{subfigure}[t]{0.24\textwidth}
    \centering
    	\includegraphics[width=\linewidth]{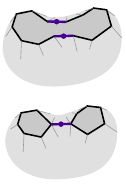}
    \caption{Case~\ref{link-across}.}
    \label{sfig:link-across}
    \end{subfigure}
    \hfill
    \begin{subfigure}[t]{0.24\textwidth}
    \centering
    	\includegraphics[width=\linewidth]{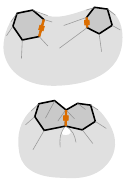}
    \caption{Case~\ref{link-genus}.}
    \label{sfig:link-genus}
    \end{subfigure}
    \caption{The four ways in which identifying a pair of arcs can change the vertex link.}
    \label{fig:link-fold-cases}
\end{figure}

It therefore suffices to give an inductive proof that after each fold, every filling petal joins two normal arcs that belong to the same boundary circle of $L$.
This is trivially true before any folds have been performed, which establishes the base case.
For the inductive step, suppose we are about to fold along a filling petal $\gamma$;
we may assume that $\gamma$ joins two normal arcs in the same boundary circle $c$ of $L$, and hence splits $c$ into two pieces.
After folding along $\gamma$, these two pieces of $c$ become two distinct boundary circles,
so we need to show that no filling petal joins a normal arc in one of these pieces with a normal arc in the other piece.
This follows as a consequence of the requirement that all the filling petals have non-crossing intersections at the vertex of the $3$-triangulation.
Thus, by induction, we conclude that the folds only ever identify normal arcs in the same boundary circle of $L$,
and hence that $L$ remains a genus-$0$ surface when the Wedge-Folder routine terminates.

For~\ref{sprop:WF-bdry}, we compare the process of attaching a folded tetrahedron along a filling petal to the process of attaching a wedge piece in the pinched filling algorithm.

    Attaching a tetrahedron by first layering over the petal, then folding along it is equivalent to
    folding the tetrahedron first, \textit{then} attaching it along the filling petal.
    Figure~\ref{fig:diagram-P-vs-C} compares this process with the analogous step in the pinched filling algorithm.

    \begin{figure}[htbp]
        \centering
        \includegraphics[width=\textwidth]{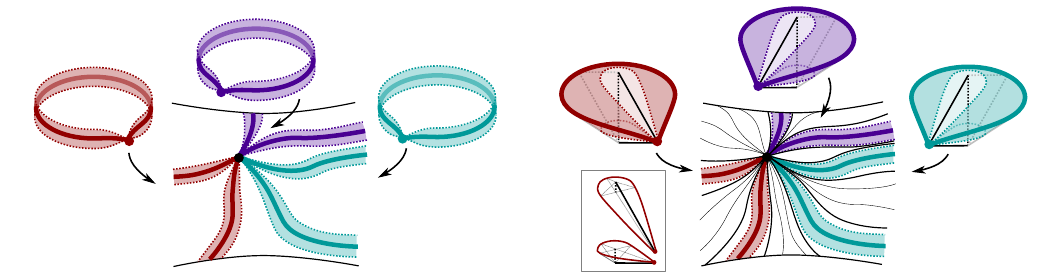}
        \caption{Comparing Step~\ref{step:standard-2} of the pinched filling algorithm (left) with
        the corresponding step in the combinatorial filling algorithm (right).
        Inset: The two cone pieces from a folded tetrahedron.}
        \label{fig:diagram-P-vs-C}
    \end{figure}

Notice that the wedge pieces can be embedded inside the folded tetrahedra.
In particular, attaching a folded tetrahedron has the same effect as attaching a wedge piece, except for some additional identification corresponding to
the \textit{cone pieces} at the vertices away from the filling petal (shown in Figure~\ref{fig:diagram-P-vs-C}, inset).
Let us consider the effect of these cone pieces.

Figure~\ref{subfig:WF-local-gluing} shows how the six corners of the tetrahedron meet around the vertex before folding, and how the gluing identifies them.
We can expand the neighbourhood of the filling petal to cover most of each of the faces,
leaving only the tips of four of the corners as in Figure~\ref{subfig:WF-local-expanded}.
If we identify the tips first, as in Figure~\ref{subfig:WF-local-cone}, we see that they just pinch off part of the 2-sphere boundary as a ball.
Hence, we can think of the extra identification as premature filling of part of the 3-ball.

    \begin{figure}[htbp]
        \centering
        \begin{subfigure}{0.325\textwidth}
        \centering
        \includegraphics[width=\textwidth]{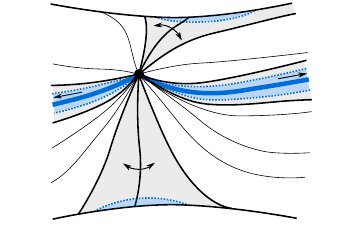}
        \caption{A local view of the six corners meeting the vertex.}
        \label{subfig:WF-local-gluing}
        \end{subfigure}
        \begin{subfigure}{0.325\textwidth}
        \centering
        \includegraphics[width=\textwidth]{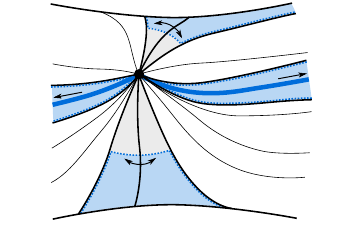}
        \caption{Expanding the neighbourhood of the filling petal.}
        \label{subfig:WF-local-expanded}
        \end{subfigure}
        \begin{subfigure}{0.325\textwidth}
        \centering
        \includegraphics[width=\textwidth]{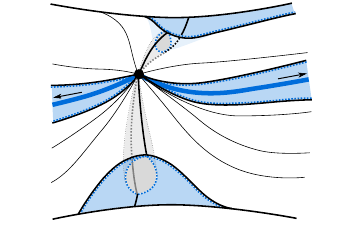}
        \caption{Pinching closed just the tips that correspond to the cone pieces.}
        \label{subfig:WF-local-cone}
        \end{subfigure}
        \caption{Understanding how the boundary is affected by attaching a folded tetrahedron.}
        \label{fig:WF-cone-pieces}
    \end{figure}

Without the cone pieces, we know from the pinched filling algorithm that the boundary is a pinched 2-sphere.
Then by considering the cone pieces alone, as in Figure~\ref{subfig:WF-local-cone}, we see that they just pinch the pinched boundary even more.
In particular, the boundary is still a pinched 2-sphere.

The fact that we still have a 2-sphere after folding along a petal can also be verified by observing that the Euler characteristic of the boundary is preserved:
two vertices are identified to become one, and two faces are glued together so that they no longer appear in the boundary,
but the resulting reduction in the Euler characteristic is exactly cancelled out by the fact that five boundary edges become two after the fold.
\end{proof}

\subsection{Ball-Filler}
After making the filling petals homotopically trivial using the Wedge-Folder routine,
we need to close up the remaining boundary faces in such a way that no new topology is introduced.
This is analogous to the step in the topological filling algorithm where a ball is glued in.
Hence, we refer to this subroutine as the Ball-Filler.

Roughly, the strategy of the Ball-Filler subroutine is to fold pairs of boundary faces together until the boundary becomes empty.
However, we cannot perform folds indiscriminately.
To see why, recall from Definition~\ref{def:folding} that each folding operation causes an off-diagonal $\gamma$ of a quadrilateral to become homotopically trivial as a loop.
Such a fold is acceptable provided that $\gamma$ is already a closed loop in the boundary $2$-triangulation;
in this case, since the boundary is a pinched $2$-sphere, $\gamma$ is already homotopically trivial before the fold, and hence the fold does not introduce new topology.
However, as discussed in Section~\ref{ssec:pinched-manifolds}, it is also possible
for $\gamma$ to be a curve whose endpoints lie on two distinct vertices of the boundary $2$-triangulation;
in this case, we must avoid folding along $\gamma$, because such a fold would add genus to the vertex link
(see the proof of part~\ref{sprop:WF-link} of Proposition~\ref{prop:WF-is-legit}, where we previously discussed how folding changes the vertex link).

Let us introduce some terminology before describing the Ball-Filler subroutine in more detail.
We call a one-vertex 3-triangulation $\mathcal{T}_0$ \emph{fillable} if its boundary consists of zero or more pinched $2$-sphere components;
in particular, the output from the Wedge-Folder subroutine is a fillable triangulation, since its boundary consists of precisely one pinched $2$-sphere.
The closed $3$-manifold $M_0$ obtained by filling in each $2$-sphere boundary component of
$\mathcal{T}_0$ with a $3$-ball is called the \emph{filled manifold} associated to $\mathcal{T}_0$.
We call another fillable triangulation $\mathcal{T}_1$ a \emph{partial filling} of $\mathcal{T}_0$ if:
\begin{itemize}
\item the filled manifolds associated to $\mathcal{T}_0$ and $\mathcal{T}_1$ are homeomorphic; and
\item $\mathcal{T}_1$ has strictly fewer boundary triangles than $\mathcal{T}_0$.
\end{itemize}
When a partial filling has empty boundary and therefore actually triangulates the filled manifold,
it would perhaps be more natural to call it a ``complete filling''; however, we use the term `partial filling' anyway, to avoid an unnecessary case distinction.

\begin{algorithm}[Ball-Filler]\label{algm:ball-filler}
    Beginning with $i=0$: for a fillable $3$-triangulation $\mathcal{T}_i$ with non-empty (and possibly disconnected) boundary $2$-triangulation $B_i$,
    iteratively construct a partial filling $\mathcal{T}_{i+1}$ as follows, until no boundary faces remain:
    \begin{enumerate}
    \item\label{step:BF-fold}
    Attempt to build a partial filling $\mathcal{T}_{i+1}$ in one of the following ways:
    	\begin{enumerate}
    	\item\label{step:BF-fold-diag}
    	If there is an off-diagonal in $B_i$ that realises a closed curve $\gamma$,
    	then build $\mathcal{T}_{i+1}$ by folding along $\gamma$.
    	\item\label{step:BF-fold-edge}
    	If there is an edge \textbf{\textit{e}} in $B_i$ that realises a closed curve $\gamma$,
    	build $\mathcal{T}_{i+1}$ by first layering a tetrahedron over \textbf{\textit{e}}, and then folding along $\gamma$.
    	\end{enumerate}
    Whenever this attempt is successful, check whether $\mathcal{T}_{i+1}$ is closed.
    If so, terminate and return $\mathcal{T}_{i+1}$;
otherwise, increase $i$ by one and go back to the start of this algorithm (in other words, restart
using $\mathcal{T}_{i+1}$ in place of $\mathcal{T}_i$).
    \item\label{step:BF-layer}
Whenever Step~\ref{step:BF-fold} is unsuccessful, modify the boundary $2$-triangulation $B_i$ as follows:
    	\begin{enumerate}
    	\item
    	Fix a vertex $v$ of $B_i$.
    	Let $d$ denote the degree of $v$, and arbitrarily label the edges of $B_i$ incident to $v$ by $e_1, \ldots, e_d$.
\item
Let $j=1$.
\item\label{step:BF-layer-ej}
Layer a tetrahedron over the edge $e_j$.
For simplicity, we continue to denote the boundary $2$-triangulation by $B_i$ after this layering.
\item\label{step:BF-check-closed-curve}
If $B_i$ contains a closed curve that is realised by either an edge or an off-diagonal, then go back to Step~\ref{step:BF-fold} (which is now guaranteed to succeed).
Otherwise, increase $j$ by one and go back to Step~\ref{step:BF-layer-ej}.
    	\end{enumerate}
    \end{enumerate}
\end{algorithm}

\begin{example}\label{eg:BF-genus-3}
    Recall Figure~\ref{fig:WF-genus-3} from Example~\ref{eg:WF-genus-3}.
    At the end of the Wedge-Folder routine, the remaining boundary was a four-vertex triangulation of a pinched 2-sphere
    (redrawn in Figure~\ref{subfig:BF-genus-3-eg.a}).
    Running Ball-Filler on this triangulation requires no additional tetrahedra, since there are two closed curves that appear as off-diagonals.

    Notice that if the original triangulation was slightly different, the 2-sphere triangulation could instead
    have been one of those shown in Figures~\ref{subfig:BF-genus-3-eg.b} or~\ref{subfig:BF-genus-3-eg.c}.
    In these cases, layering is required before folding, as dictated by
    Steps~\ref{step:BF-fold-edge} and~\ref{step:BF-layer} in the Ball-Filler algorithm.
    The pale dashed edge in Figure~\ref{subfig:BF-genus-3-eg.c} indicates a layering that would realise
    the same boundary $2$-triangulation as in Figure~\ref{subfig:BF-genus-3-eg.a}.

    \begin{figure}[htbp]
        \centering
        \begin{subfigure}{0.325\textwidth}
            \centering
            \includegraphics[width=0.75\textwidth]{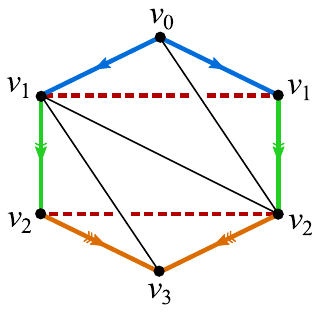}
            \caption{There are two closed curves, rooted at $v_1$ and $v_2$
            respectively, that appear as \textbf{off-diagonals}.}
            \label{subfig:BF-genus-3-eg.a}
        \end{subfigure}
        \begin{subfigure}{0.325\textwidth}
            \centering
            \includegraphics[width=0.75\textwidth]{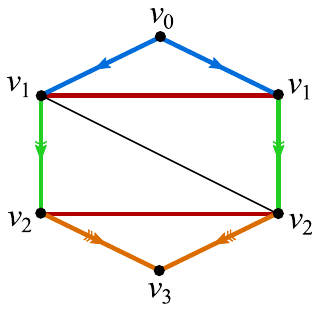}
            \caption{There are two closed curves, rooted at $v_1$ and $v_2$
            respectively, that appear as \textbf{edges}.}
            \label{subfig:BF-genus-3-eg.b}
        \end{subfigure}
        \begin{subfigure}{0.325\textwidth}
            \centering
            \includegraphics[width=0.75\textwidth]{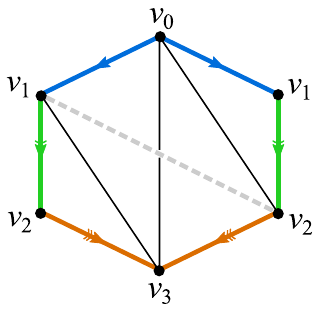}
            \caption{There are \textbf{no} closed curves that can be realised as edges or off-diagonals.}
            \label{subfig:BF-genus-3-eg.c}
        \end{subfigure}
        \caption{Examples of triangulations of the pinched 2-sphere that
        require each of the cases in Algorithm~\ref{algm:ball-filler}.}
        \label{fig:BF-genus-3-examples}
    \end{figure}
\end{example}

\begin{remark}\label{rem:BF-check-closed-curve}
As we will see in the proof of Proposition~\ref{prop:BF-terminates},
the running time of the Ball-Filler depends on how carefully we implement Step~\ref{step:BF-check-closed-curve}.
Naively, to check whether the boundary $2$-triangulation $B_i$ contains a closed curve that
is realised by either an edge or an off-diagonal, we need to iterate through $O(g)$ quadrilaterals in $B_i$.
However, this is unnecessarily inefficient, because the same check was already done immediately prior to the most recent layering performed by Step~\ref{step:BF-layer-ej}.
Thus, we only need to redo this check on quadrilaterals that include one of the two new boundary triangles introduced by the layering.
With this optimisation, Step~\ref{step:BF-check-closed-curve} becomes a constant-time operation.
\end{remark}

\begin{proposition}\label{prop:BF-terminates}
	Given a fillable triangulation with non-empty boundary, it is always possible to construct a partial filling.
    Hence, the Ball-Filler algorithm terminates in $O(g^2)$ time, and returns a triangulation of a closed $3$-manifold.
    Moreover, the Ball-Filler adds $O(g^2)$ new tetrahedra (via layering).
\end{proposition}
\begin{proof}
Let $\mathcal{T}_i$ be a fillable $3$-triangulation with non-empty boundary $B_i$.
Recall that the boundary $2$-triangulation $B_i$ might have more than one vertex,
in which case these vertices are all pinched at the single vertex of the $3$-triangulation $\mathcal{T}_i$.

The goal of the Ball-Filler routine is to eventually close the boundary using only layering and folding.
As discussed earlier, to ensure that the vertex link in $\mathcal{T}_i$ remains a genus-$0$ surface,
we must only fold along off-diagonals that begin and end at the same vertex of $B_i$.
As shown in Example~\ref{eg:BF-genus-3}, such off-diagonals might not immediately exist,
so the Ball-Filler often needs to modify the boundary $2$-triangulation by flipping edges.
To keep notation simple, we will continue to use $B_i$ to denote any boundary $2$-triangulation obtained by a sequence of such flips.

To prove that it is always possible to build a partial filling of $\mathcal{T}_i$,
we will show that there is always a finite sequence of flips on boundary edges such that $B_i$ eventually contains
at least one quadrilateral whose off-diagonal begins and ends at the same vertex of $B_i$ (as opposed to two distinct vertices that are pinched together in $\mathcal{T}_i$).
Specifically, we will prove that this is achieved by the edge flips specified in Step~\ref{step:BF-layer} of Algorithm~\ref{algm:ball-filler}.

We first eliminate the following cases where Step~\ref{step:BF-fold} immediately succeeds in building a partial filling of $\mathcal{T}_i$:
\begin{itemize}
\item If $B_i$ contains a vertex of degree 1 or 2, then it must contain a closed curve realised by
an edge or an off-diagonal, respectively (see Figure~\ref{fig:BF-degree1and2}).
\item If $B_i$ contains a triangle with more than one corner identified to a single vertex, then it must contain a closed curve realised by an edge
(all the possible cases are summarised in Figure~\ref{fig:BF-multi-contribution}).
\end{itemize}

\begin{figure}[htbp]
    \centering
    \begin{subfigure}{0.48\textwidth}
    \centering
\begin{tikzpicture}[inner sep=0pt]

\node at (0,0) {
	\includegraphics[width=0.5\textwidth]{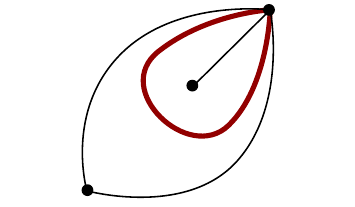}
};

\node at (0.325,0.1) {$v$};
\node at (-0.275,-0.3) {$e$};

\end{tikzpicture}
    \caption{A degree $1$ vertex $v$, with one edge $e$ (bold red) around it forming a closed curve.
    Layering over $\textbf{\textit{e}}$ results in the situation in~\ref{subfig:BF-degree2}.}
    \label{subfig:BF-degree1}
    \end{subfigure}
    \begin{subfigure}{0.48\textwidth}
    \centering
\begin{tikzpicture}[inner sep=0pt]

\node at (0,0) {
	\includegraphics[width=0.5\textwidth]{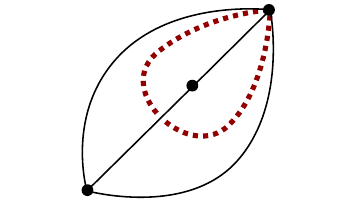}
};

\node at (0.375,0.1) {$w$};
\node at (0.1,-0.575) {$\gamma$};

\end{tikzpicture}
    \caption{A degree 2 vertex $w$, with a closed curve $\gamma$ (dotted red) appearing as an off-diagonal.
    These two faces can be identified by folding along $\gamma$.}
    \label{subfig:BF-degree2}
    \end{subfigure}
    \caption{If there is a vertex of degree 1 or 2, then we can build a partial filling after layering at most one extra tetrahedron.}
    \label{fig:BF-degree1and2}
\end{figure}

\begin{figure}[htbp]
    \centering
    \includegraphics[width=0.65\textwidth]{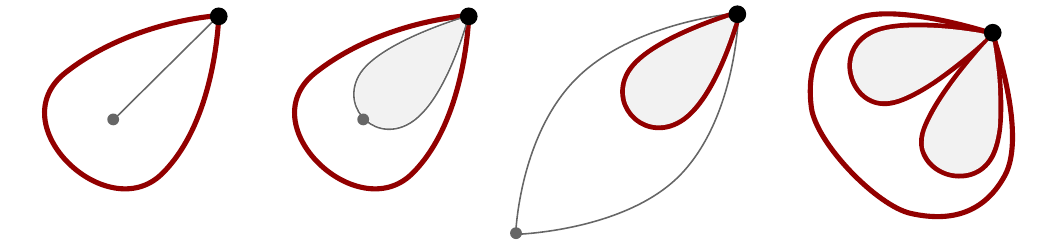}
    \caption{A single triangle that contributes more than one of its corners to the degree of a vertex must be in one the four configurations shown. In each case there is at least one edge (indicated using bold red lines) that forms a closed curve.}
    \label{fig:BF-multi-contribution}
\end{figure}

With this in mind, consider the case where Step~\ref{step:BF-fold} does not immediately succeed in building a partial filling.
In this case, we know in particular that every vertex of $B_i$ has degree at least 3,
and that no triangle in $B_i$ contributes more than once to the degree of any vertex.
Fixing an arbitrary vertex $v$ of $B_i$, and letting $d$ denote the degree of $v$, we therefore know that $B_i$ contains exactly $d$ distinct triangles meeting $v$.
Thus, flipping edges incident to $v$, as in Step~\ref{step:BF-layer} of the Ball-Filler routine,
iteratively reduces the degree of $v$ (see Figure~\ref{fig:BF-decreasing-degree}).
Observe that this is guaranteed to lead to a scenario where Step~\ref{step:BF-fold} will successfully build a partial filling;
in the worst case, we might need to reduce the degree of $v$ all the way down to 2,
but at that point $B_i$ would contain an off-diagonal that realises a closed curve, and hence no further flipping would be needed.

\begin{figure}[htbp]
    \centering
    \includegraphics[width=0.6\textwidth]{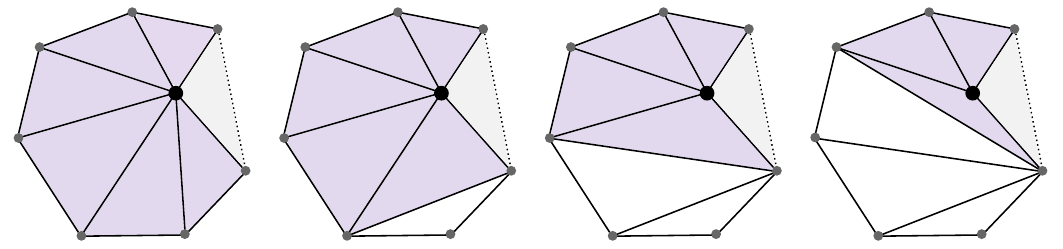}
    \caption{When no triangle contributes more than one of its corners to the degree of a vertex,
    they must be arranged as shown, and the degree can be reduced by layering.}
    \label{fig:BF-decreasing-degree}
\end{figure}

Finally, notice the effect of the Ball-Filler on the number of boundary faces:
any layering covers two boundary faces but replaces them with two new ones, so the number of boundary faces is unchanged;
on the other hand, any fold identifies two boundary faces to each other, so the number of boundary faces decreases by two.
Therefore, each time we go from a $3$-triangulation $\mathcal{T}_i$ to a partial filling $\mathcal{T}_{i+1}$, we decrease the number of boundary faces.
Hence, the Ball-Filler eventually terminates, at which point we have a $3$-triangulation with empty boundary.

To complete the proof, let us examine the computational complexity of the Ball-Filler.
At worst, Step~\ref{step:BF-fold} iterates through $O(g)$ quadrilaterals to determine whether
either~\ref{step:BF-fold-diag} or~\ref{step:BF-fold-edge} succeeds in building a partial filling.
If this is unsuccessful, then we need to run Step~\ref{step:BF-layer}, which could layer on $O(g)$ new tetrahedra in the worst case;
we need to run Step~\ref{step:BF-check-closed-curve} after each such layering,
but this is constant-time if we use the optimisation discussed in Remark~\ref{rem:BF-check-closed-curve}.
Once Step~\ref{step:BF-layer} is done, we repeat Step~\ref{step:BF-fold}, which is now guaranteed to succeed.
So, altogether, we can build a single partial filling in $O(g)$ time.
Since each partial filling reduces the number of boundary faces by two, and since we initially have $O(g)$ boundary faces, the Ball-Filler runs in $O(g^2)$ time overall;
moreover, each partial filling requires adding $O(g)$ new tetrahedra in the worst case, so in total the Ball-Filler adds $O(g^2)$ new tetrahedra.
\end{proof}

\subsection{Full algorithm}\label{subsec:full-algm}
Putting all of the subroutines together, we get our main algorithm, which allows us to
fill in the boundary component of any one-vertex triangulation of a $3$-manifold with boundary.
In particular, by filling in the boundary of a handlebody, we can effectively construct a one-vertex triangulation of a $3$-manifold from a Heegaard diagram.

Using the notation from Section~\ref{sec:standard-diagram}, suppose $(M,B,\gamma)$ is a topological filling diagram for $M(\gamma)$, defined by a compact orientable $3$-manifold $M$ with
a single genus-$g$ boundary component $B$, together with $g$ attaching circles given by $\gamma$.
Let $\mathcal{T}$ be a one-vertex triangulation of $M$, with $B_\Delta$ the induced triangulation on $B$.
Let $\mathcal{R}$ and $\mathcal{W}$ encode the curves $\gamma$ on $B_\Delta$.

\begin{algorithm}[Combinatorial filling algorithm]\label{algm:main-combinatorial}
Given a combinatorial filling diagram $(\mathcal{T},\mathcal{R},\mathcal{W})$,
as described above, build a one-vertex triangulation of $M(\gamma)$ as follows:
\begin{enumerate}[label={(\arabic*)}]
    \item \textit{Run Petal-Resolver.}
    Layer tetrahedra over reducible edges until all filling petals appear as (boundary) edges in the triangulation.
    \item \textit{Run Quad-Isolator.}
    Layer tetrahedra until all the filling petals form diagonals of mutually disjoint quadrilaterals.
    \item \textit{Run Wedge-Folder.}
    For each filling petal, layer a tetrahedron over it to make it an off-diagonal, and then fold to make the filling petal homotopically trivial.
    \item \textit{Run Ball-Filler.}
    Iteratively identify remaining boundary faces by folding without introducing new topology.
    Terminate and return the resulting triangulation.
\end{enumerate}
\end{algorithm}

\begin{theorem}\label{thm:main-alg}
    The combinatorial filling algorithm terminates in $O(g\tau+g^2)$ time, and returns
    a one-vertex triangulation of the filled manifold $M(\gamma)$ with $O(n+g+\tau)$ tetrahedra.
\end{theorem}

\begin{proof}
    We have seen that each subroutine terminates and yields output that is valid as input for the subsequent subroutine.
By adding up the running times for the four subroutines, we see that the overall algorithm is guaranteed to terminate in $O(g\tau+g^2)$ time.

    The resulting triangulation is one-vertex by construction.
In total, the four subroutines add $O(g+\tau)$ tetrahedra to the $n$-tetrahedron input triangulation, which gives the stated bound on the size of the output triangulation.

    It remains to be shown that the output is indeed a valid triangulation of the manifold $M(\gamma)$ described by the input.
    Recall from Proposition~\ref{prop:algm-equiv-S=P} that the pinched filling algorithm constructs the same manifold as the topological filling algorithm.
    Here we show that the combinatorial filling algorithm builds the same manifold as the pinched filling algorithm.

    Recall (from Proposition~\ref{prop:diag-equiv-S=C}) that
    a topological filling diagram can be converted to a combinatorial one by:
    assigning a one-vertex triangulation to the manifold $M$, and isotoping the attaching circles so that
    they are in rooted normal position with respect to the induced boundary triangulation $B_\Delta$.
    In particular, by ignoring the triangulation, this also determines a pinched filling diagram.

    Since Petal-Resolver and Quad-Isolator only involve layering,
    the effect of these subroutines on the combinatorial filling diagram is just to change the boundary triangulation.
    Since the pinched filling diagram does not see the triangulation,
    the pinched and combinatorial filling diagrams are still equivalent by the end of the Quad-Isolator routine.

	From here, the pinched filling algorithm proceeds by attaching a collection $W$ of wedges to turn the boundary into a pinched $2$-sphere,
	then filling the pinched $2$-sphere with a $3$-ball $Z$ to obtain the $3$-manifold $M(\gamma)$.
	By construction, the triangulation at the end of the Quad-Isolator routine coincides with the closure of $M(\gamma) - (W\cup Z)$;
	in other words, this triangulation embeds inside $M(\gamma)$.
	As described in the proof of Proposition~\ref{prop:WF-is-legit}, the Wedge-Folder routine effectively attaches a collection $F$
	of folded tetrahedra, and we can view $F$ as a union of the wedges $W$ with some cone pieces.
	We can think of the cone pieces as forming part of the ball $Z$, so that our triangulation
	is still embedded inside $M(\gamma)$ at the end of the Wedge-Folder routine;
	in fact, our triangulation coincides with $M(\gamma)$ minus the interior of a pinched ball.

	The final step of the combinatorial filling algorithm is to use the Ball-Filler
	routine to close the pinched $2$-sphere boundary of our triangulation.
	This involves layering and folding operations that can be performed inside $M(\gamma)$,
	preserving the property that the triangulation is embedded in $M(\gamma)$.
	At the end of the Ball-Filler routine, our triangulation coincides with the entirety of $M(\gamma)$, and is therefore homeomorphic to $M(\gamma)$.
\end{proof}

\section{Triangulation complexity}\label{sec:complexity}

Throughout the rest of this paper, we will refer to the triangulations constructed by
Algorithm~\ref{algm:main-combinatorial} as \emph{petal-filled} triangulations.
In this section, we consider the complexity of petal-filled triangulations,
particularly in the case where the input filling diagram describes a Heegaard splitting.

\subsection{Cutwidth}
\begin{definition}[Cutwidth] \label{defn:cutwidth}
Let $G=(V,E)$ be a multigraph.
For a given ordering $(v_1,\ldots,v_n)$ of the vertices in $V$, define a \emph{cutset} $C_\ell$ for each $\ell\in\{1,\ldots,n-1\}$ to be the set of edges joining all vertices $v_i$ for $i\leq\ell$ to any vertices $v_j$ with $j>\ell$.
The \emph{width} of the ordering is given by the maximum of $\lvert C_\ell \rvert$ across all choices for $\ell$.
The \emph{cutwidth} of $G$ is the minimum width across all possible orderings of $V$.
The cutwidth of a triangulation $\mathcal{T}$ is defined to be the cutwidth of the dual graph of $\mathcal{T}$.
\end{definition}

\begin{remark}
    We can find an upper bound on cutwidth by calculating the width of a single choice of ordering of tetrahedra
    (which corresponds to an ordering of the vertices of the dual graph).
\end{remark}

\begin{theorem}[Genus bounds cutwidth] \label{thm:bounded-cutwidth}
    When the input triangulation for Algorithm~\ref{algm:main-combinatorial} is a layered handlebody of genus $g$,
    the petal-filled triangulation obtained as output has cutwidth bounded above by $4g-2$.
\end{theorem}
\begin{proof}
    The boundary of the initial layered handlebody consists of $4g-2$ triangles.
This number of boundary triangles is preserved as we layer tetrahedra onto the boundary.

    Layering tetrahedra induces a natural ordering of tetrahedra in a given triangulation.
    Starting with the layered handlebody, we label the tetrahedra in the order they are layered.

	Our algorithm modifies the given triangulation via a sequence of layering and folding.
	We continue numbering the newly-added tetrahedra according to the order in which they are layered.
	Let $\Delta_i$ denote the $i$th tetrahedron in this ordering.

	At every stage, we have at most $4g-2$ boundary triangles, since layering preserves this number and folding reduces this number by two.
	Consider the $\ell$th cutset $C_\ell$ under this ordering;
	this corresponds to gluings between tetrahedra $\Delta_i$ and $\Delta_j$ for some $i\leqslant\ell<j$.
To bound $C_\ell$, consider the boundary faces that remain after the $\ell$th layering.
The size of $C_\ell$ is equal to the number of such faces that are eventually glued to some tetrahedron $\Delta_j$, where $j>\ell$.
Since there are at most $4g-2$ boundary faces, the cutset $C_\ell$ must have size at most $4g-2$.\footnote{%
The size of $C_\ell$ need not be equal to $4g-2$ because some pairs of boundary faces might be glued to each other via folding;
such gluings do not contribute to $C_\ell$.}
\end{proof}

In Section~\ref{sec:experimentation}, we discuss cutwidths for different triangulations of the same manifolds;
in particular, we construct many petal-filled triangulations, and examine how the cutwidth changes after simplifying using Regina~\cite{Regina}.

\subsection{Minimal triangulations}\label{sec:min-tris}

A triangulation $\mathcal{T}$ of a closed $3$-manifold $M$ is \textit{minimal} if there is no triangulation of $M$ with fewer tetrahedra than $\mathcal{T}$.
Layered solid tori frequently appear as subcomplexes in minimal triangulations;
this is discussed, for example, in Burton's thesis~\cite[Chapter~3]{Burton2003Thesis},
and can also be seen in explicit examples by studying the ``composition'' of census triangulations in Regina.
Despite this, layered triangulations are only conclusively known to be minimal for two specific families of lens spaces (i.e., $3$-manifolds with Heegaard genus equal to $1$):
$L(2n,1)$~\cite{JRT2009MinimalLens} and $L(4n,2n-1)$~\cite{JRT2011CoveringsMinimal}, for $n\geqslant 2$.
As for layered handlebodies of genus $g>1$, since Regina does not identify higher genus
layered handlebodies as subcomplexes, there is also less we can say about their prevalence.

One can, however, define the notion of a \textit{minimal layered} triangulation to be a layered triangulation that corresponds to
the shortest path in the flip graph between two given boundary triangulations.
In genus $1$, such a path is unique, but this is not true for higher genus.

\begin{proposition}
    In the case where the boundary component to be filled has genus $1$---in which case, a filling diagram just describes
    a Dehn filling---Algorithm~\ref{algm:main-combinatorial} constructs the standard layered solid torus that realises the desired Dehn filling.
    The sequence of layerings corresponds to the unique non-backtracking path between two triangulations in the flip graph.
\end{proposition}
\begin{proof}
Layered solid tori are commonly described by the edge weights of the meridian curve in (unrooted) normal position.
In detail, LST$(a,b,c)$ describes the layered solid torus in which the unrooted meridian curve
intersects the three boundary edges in $a$, $b$ and $c$ points, respectively.
Without loss of generality, we can assume $a\leqslant b\leqslant c$.
It is well-known that $a+b=c$ (for instance,
see~\cite[Section~1.2]{Burton2003Thesis} or~\cite[Section~4.1]{JacoRubinstein2006}),
and by layering over each edge, the resulting layered solid torus is:
LST$(b,c,b+c)$ if the edge with weight $a$ was covered;
LST$(a,c,a+c)$ if the edge with weight $b$ was covered;
and LST$(a,b,b-a)$ if the edge with weight $c$ was covered.
In particular, the only layering that results in a reduction of total edge weight is the layering that covers the maximal-weight edge.

To see that the same logic holds when the meridian is rooted, we can do everything in the universal cover,
and observe that changing the position of the meridian by translations has no effect on the procedure.

Hence, our algorithm never needs to make any arbitrary choices when $g=1$,
and it will always build the layered solid torus corresponding to the unique shortest path in the flip graph.
\end{proof}

The situation for genus $2$ and higher is more complicated.
There are no longer unique paths between any pair of boundary triangulations,
and there is not a unique 2-triangulation that realises the filling petals as edges.
It would be an interesting challenge to try to improve our algorithm by analysing paths in the flip graph and
optimising the number of tetrahedra we use in the construction of handlebodies of genus $g>1$.

Beyond optimising the algorithm itself, there is also the question of how to optimally encode the input.
In particular, given a system of attaching circles, we have many choices for the ambient triangulation and for the rooted normal curves.
Which choice gives the minimum total edge weight?
This question arises naturally because our algorithm constructs a triangulation whose size scales, in the worst case, linearly with the total edge weight in the input.

It is also worth considering a related line of inquiry.
As mentioned in Section~\ref{sec:intro}, Husz\'ar and Spreer~\cite[Theorem~27]{HuszarSpreer2018WidthFull} used work
from Bell's thesis~\cite[Section~2.4]{Bell2015Thesis} to show that there is an algorithm to convert a Heegaard splitting,
written as a word $w$ encoding a composition of Dehn twists, into a layered triangulation with $O\left( g+K|w| \right)$ tetrahedra, where:
\begin{itemize}
\item
$g$ is the genus of the Heegaard splitting;
\item
$|w|$ is the length of the word $w$; and
\item
$K$ is a constant depending only on $g$ and the set of Dehn twists used to generate the mapping class group.
\end{itemize}
In this context, there is a natural question which is analogous to our previous question about edge weights:
what is the minimum length of the word $w$?
See~\cite[Appendix~C]{HuszarSpreer2018WidthFull} and the references therein for more details and discussion.

\section{Implementation}\label{sec:implementation}

\subsection{Using the algorithm}\label{sec:getting-started}
To start using our implementation, download the main Python script \texttt{heegaardbuilder.py} and the auxiliary file \texttt{heegaarderror.py}
from \url{https://github.com/AlexHe98/heegaardbuilder/},
and save these two Python scripts in the same directory.
Our implementation relies on Regina~\cite{Regina}, so Regina will need to be installed as well (if it is not already).
Here are a couple of ways in which users can run our code in an interactive Python session:
\begin{itemize}
\item
Users can import \texttt{heegaardbuilder.py} in the Python console provided in Regina's graphical user interface.
Note that for this to work, the working directory needs to be set to the directory containing \texttt{heegaardbuilder.py}.
\item
For users not on Windows, the implementation in \texttt{heegaardbuilder.py} can also
be accessed by running \texttt{regina-python} in interactive mode.
For example, if \texttt{regina-python} is in the directory \texttt{/usr/local/bin},
and assuming that \texttt{heegaardbuilder.py} is in the current directory,
then the command to run \texttt{heegaardbuilder.py} in an interactive session is:
\begin{verbatim}
/usr/local/bin/regina-python -i heegaardbuilder.py
\end{verbatim}
\end{itemize}
In such an interactive session, the general scheme for running our algorithm is as follows:

\begin{example}\label{eg:python}
To build a petal-filled triangulation using our algorithm, start by setting \texttt{tri}
to be any one-vertex 3-triangulation with a single genus-$g$ boundary component.
For example, use a layered handlebody of genus $g$ (Regina can build one for you).
Set \texttt{weights} to be an $n$-tuple, where $n$ is the total number of edges in \texttt{tri};
entries for boundary edges correspond to the weights on those edges, and entries for non-boundary edges must be 0.
Set \texttt{resolved} to be a (possibly empty) set of edge labels corresponding to pre-resolved edges.
The following sequence of commands can be used to build the petal-filled triangulation determined by the input.
\begin{verbatim}
H = HeegaardBuilder()
H.setBouquet(Triangulation3(tri), weights, resolved)
H.resolveGreedily()
H.fillHandlebody()
filled = H.triangulation()
\end{verbatim}
\end{example}

\subsection{Some remarks on our implementation}
\subsubsection{Invalid edge weights}
    Recall (from Remark~\ref{rmk:invalid-input}) that it is possible to enter
    edge weights describing a collection of curves that is invalid as a filling bouquet.
    Whenever the input does not describe a valid filling bouquet, our implementation detects this and provides a detailed error message.
    For example:
    \begin{itemize}
        \item \textit{Edge weights fail to satisfy the matching constraints in triangle X}.
        This error message arises when the input fails to satisfy the constraints in Section~\ref{sec:weights=curves}.
        \item \textit{Edges X and Y form a pair of resolved filling petals that meet transversely}.
        This error message arises when two filling petals intersect transversely at the vertex.
        \item \textit{After cutting along the filling bouquet, the boundary surface splits into X components}.
        This error message arises when the collection of attaching circles is separating.
        \item \textit{A normal curve was left over after resolving all filling petals}.
        This error message arises when the edge weights describe an unrooted curve.
    \end{itemize}

\subsubsection{Order of layering}

Whenever there is more than one reducible edge, our algorithm works regardless of which such edge we choose to layer over
(in particular, there is no need to follow the proof of Proposition~\ref{prop:PR-terminates} in considering only maximal-weight reducible edges).
Our implementation currently offers the following options for determining how this choice is made:
\begin{description}
\item[\texttt{HeegaardBuilder.resolveAllPetals()}]\hfill\par
This option essentially chooses reducible edges arbitrarily.
In the current implementation, the chosen edge is simply the first reducible edge with respect to the edge ordering inherited from the Regina triangulation.
\item[\texttt{HeegaardBuilder.resolveGreedily()}]\hfill\par
This option iterates through all reducible edges, and chooses an edge such that layering reduces the total weight as much as possible.
If there are multiple choices, the algorithm chooses one such edge arbitrarily.
\item[\texttt{HeegaardBuilder.resolveInAllWays(greedy=True)}]\hfill\par
This option branches and follows every possible choice.
If \texttt{greedy} is set to \texttt{True}, the choices are restricted to the edges that reduce the total weight the most;
otherwise, if \texttt{greedy} is set to \texttt{False}, the algorithm is allowed to layer across any reducible edge.
\end{description}

\subsubsection{Removing tetrahedra}
    Any time an edge is required to be flipped, our implementation first checks if the flip can be realised by
    removing a tetrahedron rather than layering on a new one unnecessarily.
    This corresponds to avoiding backtracking in the flip graph.
    Note that the same triangulation could be obtained by running
    the algorithm as stated (in Section~\ref{sec:main-algorithm}),
    and then simplifying the final triangulation by using 2-0 moves to eliminate any pairs of tetrahedra where backtracking occurred.

\subsection{Examples}\label{sec:examples}
To give explicit examples, we specify triangulations using \textit{isomorphism signatures}, which are strings of letters that
uniquely encode a triangulation up to \textit{combinatorial isomorphism} (i.e., up to reordering tetrahedra and/or relabelling vertices).
Isomorphism signatures were introduced by Burton;
see Section~3.2 of either~\cite{Burton2011SoCG} or~\cite{Burton2011arXiv}.
The software Regina~\cite{Regina} has
built-in functionality to convert back and forth between triangulations and isomorphism signatures.

Here we use our algorithm to construct one-vertex petal-filled triangulations corresponding to some simple genus-$2$ Heegaard splittings.
Recall that we apply our algorithm to Heegaard diagrams (in the sense of Setting~\ref{settingB}) by assuming that one side of the diagram has already been filled,
so that the first set of attaching circles corresponds to the meridians of an input handlebody $M$, and the second set corresponds to filling curves on $\partial M$.
Figure~\ref{fig:topHeegaardExamples} shows the filling curves that we use to represent
Heegaard splittings of $S^3$, $(S^2\times S^1)\# (S^2\times S^1)$ and $L(3,1)$.

\begin{figure}[htbp]
    \centering
    \includegraphics[width=\textwidth]{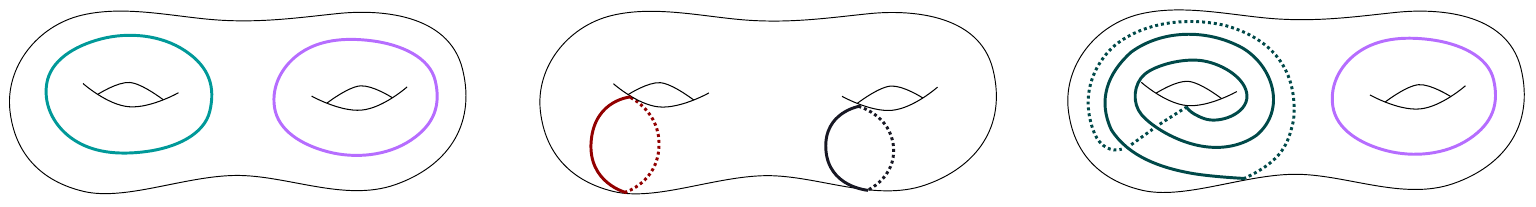}
    \caption{Topological filling diagrams for (left to right): $S^3$, $(S^2\times S^1)\# (S^2\times S^1)$, and $L(3,1)$}
    \label{fig:topHeegaardExamples}
\end{figure}

To convert the topological filling diagrams of Figure~\ref{fig:topHeegaardExamples} into combinatorial filling diagrams,
we endow the input handlebody $M$ with the triangulation given by the isomorphism signature \texttt{eHuGabdes}.
This triangulation is (up to combinatorial isomorphism) the default orientable genus-$2$ handlebody built by Regina,
and it is a minimal layered handlebody consisting of 4 tetrahedra.
We describe this triangulation in detail in Appendix~\ref{app:basic-triangulations}.

Since we are assuming that one set of attaching circles (of the Heegaard diagram) corresponds to the meridians of
the starting handlebody, we must locate the meridians of $M$ with respect to the triangulation.
This is done explicitly in Appendix~\ref{app:basic-triangulations},
and the corresponding embedding of the boundary triangulation into $\mathbb{R}^3$ is shown in Figure~\ref{fig:eHuG-for-examples}.

\begin{figure}[htbp]
    \centering
    \includegraphics[width=0.9\textwidth]{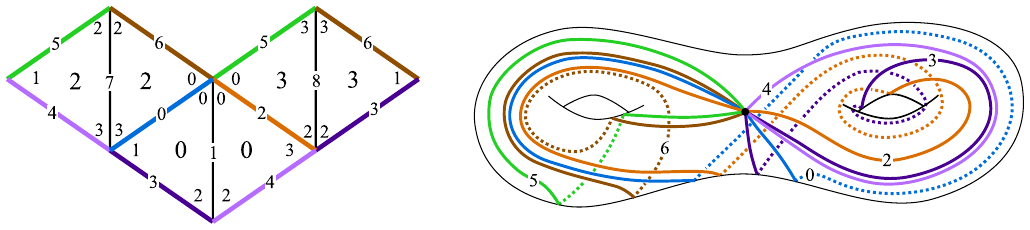}
    \caption{Left: The boundary triangulation of \texttt{eHuGabdes},
    using labels from Regina (as in Appendix~\ref{app:basic-triangulations}).
    Note that there are three types of labels:
    tetrahedron number (in the centre of each face), edge number (in the centre of each edge), and vertex number (in the corners of each face).
    Right: An embedding of this marked genus-$2$ boundary surface into $\mathbb{R}^3$, with edges 1, 7 and 8 omitted for simplicity.}
    \label{fig:eHuG-for-examples}
\end{figure}

For each topological filling diagram in Figure~\ref{fig:topHeegaardExamples}, we copy the filling curves onto the triangulation in Figure~\ref{fig:eHuG-for-examples},
and then isotope these circles so that they are rooted and compatible with the boundary triangulation.
We can then determine which edges correspond to pre-resolved attaching circles, and can also read off the edge-weight tuple.
This gives all the required data to input into our algorithm.

We set \texttt{tri} to be the genus-2 layered handlebody with isomorphism signature \texttt{eHuGabdes},
and run the sequence of commands from Example~\ref{eg:python} for each example filling diagram.

\subsubsection{A genus-\texorpdfstring{$2$}{2} splitting of \texorpdfstring{$S^3$}{S3}}
Figure~\ref{fig:S3comb} visualises a combinatorial filling diagram for $S^3$.
\begin{figure}[htbp]
    \centering
    \includegraphics[width=0.9\textwidth]{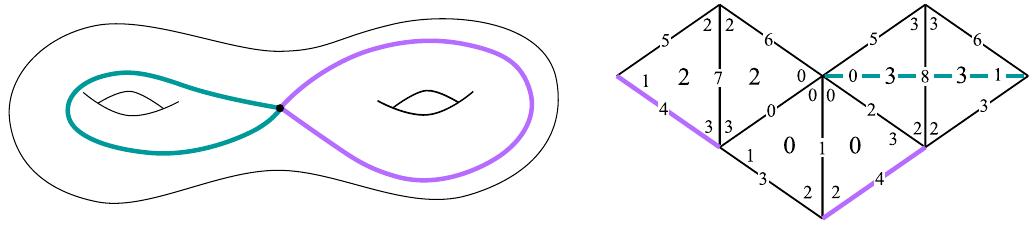}
    \caption{The rooted version of the topological filling diagram for $S^3$ (left), and how these filling curves appear with respect to the boundary triangulation (right).
    From this, we know to set \texttt{resolved=\{4\}} and \texttt{weights=(0,0,0,0,0,0,0,0,1)}.}.
    \label{fig:S3comb}
\end{figure}

Having set \texttt{weights=(0,0,0,0,0,0,0,0,1)} and \texttt{resolved=\{4\}}, we run the algorithm following the scheme from Example~\ref{eg:python}.
We then call \texttt{filled.isoSig()}, to find that the isomorphism signature is \texttt{gLLAQaddefefbpupcgo}.
To confirm that we have built a triangulation of $S^3$, we run \texttt{filled.isSphere()}, which returns \texttt{True}, as expected.
Here is the full sequence of commands to replicate this example:
\begin{verbatim}
sig = "eHuGabdes"
tri = Triangulation3(sig)
weights = (0,0,0,0,0,0,0,0,1)
resolved = {4}
H = HeegaardBuilder()
H.setBouquet(Triangulation3(tri), weights, resolved)
H.resolveGreedily()
H.fillHandlebody()
filled = H.triangulation()
filled.isoSig()
filled.isSphere()
\end{verbatim}

\subsubsection{A genus-\texorpdfstring{$2$}{2} splitting of \texorpdfstring{$(S^2\times S^1)\# (S^2\times S^1)$}{(S2 x S1) \# (S2 x S1)}}
Figure~\ref{fig:S2xS1comb} visualises a combinatorial filling diagram for $(S^2\times S^1)\# (S^2\times S^1)$.
\begin{figure}[htbp]
    \centering
    \includegraphics[width=0.9\textwidth]{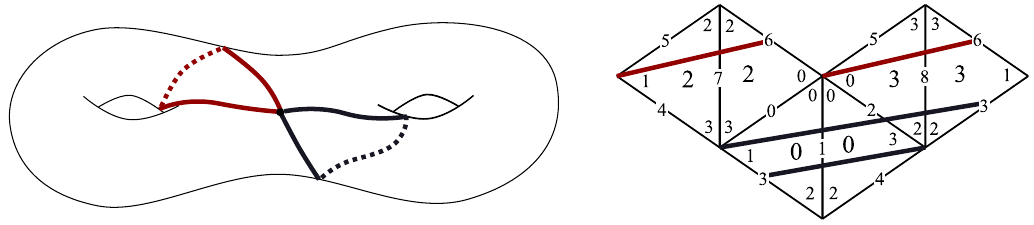}
    \caption{The rooted version of the topological filling diagram for $(S^2\times S^1)\# (S^2\times S^1)$ (left), and how these filling curves appear with respect to the boundary triangulation (right).
    From this, we know to set \texttt{resolved} to be empty and \texttt{weights=(0,2,1,1,0,0,1,1,2)}.}.
    \label{fig:S2xS1comb}
\end{figure}

Having set \texttt{weights=(0,2,1,1,0,0,1,1,2)} and \texttt{resolved} to be empty, we run the algorithm,
followed by \texttt{summands = filled.summands()}, which tells us that \texttt{filled} has two summands in its connected sum decomposition.
We can check the isomorphism signatures of the two summands using \texttt{summands[0].isoSig()} and \texttt{summands[1].isoSig()}.
Both return the isomorphism signature \texttt{cMcabbjaj}, which represents a triangulation of $S^2\times S^1$.
This confirms that \texttt{filled} is a triangulation of $(S^2\times S^1) \# (S^2 \times S^1)$.
Here is the full sequence of commands for this example:
\begin{verbatim}
sig = "eHuGabdes"
tri = Triangulation3(sig)
weights = (0,2,1,1,0,0,1,1,2)
resolved = set()
H = HeegaardBuilder()
H.setBouquet(Triangulation3(tri), weights, resolved)
H.resolveGreedily()
H.fillHandlebody()
filled = H.triangulation()
summands = filled.summands()
summands[0].isoSig()
summands[1].isoSig()
\end{verbatim}

\subsubsection{A genus-\texorpdfstring{$2$}{2} splitting of the lens space \texorpdfstring{$L(3,1)$}{L(3,1)}}
Figure~\ref{fig:lenscomb} visualises a combinatorial filling diagram for $L(3,1)$.
\begin{figure}[htbp]
    \centering
    \includegraphics[width=0.9\textwidth]{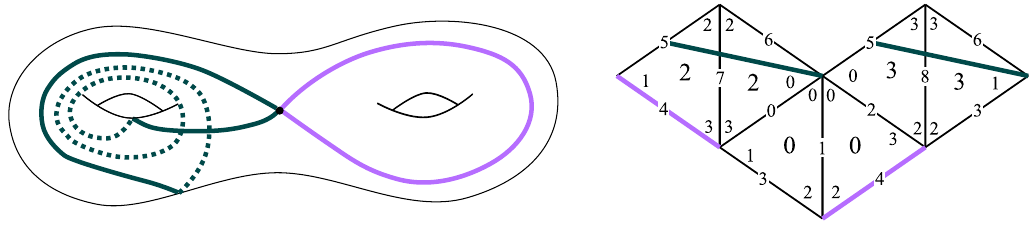}
    \caption{The rooted version of the topological filling diagram for $L(3,1)$ (left), and how these filling curves appear with respect to the boundary triangulation (right).
    From this, we know to set \texttt{resolved=\{4\}} and \texttt{weights=(0,0,0,0,0,1,0,1,1)}.}.
    \label{fig:lenscomb}
\end{figure}

Having set \texttt{weights=(0,0,0,0,0,1,0,1,1)} and \texttt{resolved=\{4\}}, we run the algorithm,
followed by \texttt{filled.isoSig()}, which returns the isomorphism signature
\[
\texttt{hLLAAkcdeeffggawadaudg}.
\]
After simplifying this filled triangulation, we can compare it against Regina's census of $3$-manifold triangulations to see that the filled manifold is the lens space $L(3,1)$.
Here is the full sequence of commands for this example:\footnote{%
This example uses the \texttt{Triangulation3.intelligentSimplify()} command.
In version~7.4 of Regina, this command will be renamed to \texttt{Triangulation3.simplify()};
the old name will remain only as a deprecated alias in Regina~7.4, and future versions of Regina will likely remove the old name entirely.}
\begin{verbatim}
sig = "eHuGabdes"
tri = Triangulation3(sig)
weights = (0,0,0,0,0,1,0,1,1)
resolved = {4}
H = HeegaardBuilder()
H.setBouquet(Triangulation3(tri), weights, resolved)
H.resolveGreedily()
H.fillHandlebody()
filled = H.triangulation()
filled.isoSig()
filled.intelligentSimplify()
Census.lookup(filled)
\end{verbatim}

\section{Experimentation}\label{sec:experimentation}
In this section we provide an overview of some 3-manifolds constructed using
Algorithm~\ref{algm:main-combinatorial} with layered handlebodies of genus 2 and 3 as inputs.
We restrict our focus to prime manifolds resulting from our construction.
We resolve the curves using \texttt{HeegaardBuilder.resolveAllPetals()}.

\subsection{Genus 2}
For the genus $2$ experiments, we used the layered triangulation of a genus-$2$ handlebody with isomorphism signature \texttt{eHuGabdes}.
This contains 9 edges, all of which are boundary edges.

Using this handlebody as input for our algorithm, we generated petal-filled triangulations of closed orientable
3-manifolds by enumerating all possible filling bouquets with total weights between 2 and 40.
We performed this enumeration in order of increasing total weight, generating weight vectors of the same total weight in lexicographical order.
We built a total of \num{115237} petal-filled triangulations of prime 3-manifolds in this way.

Note that there is no reason to expect uniqueness of the manifolds we generated.
For example, the filling bouquets described by weights \texttt{(0,0,0,0,1,0,1,0,0)} and
\texttt{(0,0,1,0,0,1,1,0,0)} both generated triangulations of $S^2\times S^1$.
After simplifying the \num{115237} petal-filled triangulations that we obtained, and removing duplicate triangulations (up to combinatorial isomorphism),
we were left with \num{38855} different simplified triangulations, \num{7804} of which were successfully identified as distinct minimal triangulations in
the Regina closed orientable 3-manifold census~\cite{Burton2011ISSAC} and Hodgson-Weeks closed hyperbolic census~\cite{HodgsonWeeks1994}.
Amongst these \num{7804} triangulations, we saw \num{3006} different census manifolds.

From the \num{3006} census manifolds, we identified 509 lens spaces.
In particular, we conclude that for the other \num{2497} census manifolds that we found, the Heegaard genus is equal to 2.
For the \num{31051} unidentified triangulations, the most we can say is that they have Heegaard genus at most 2.

Delving a little deeper into the topology of the triangulations that we found,
we identified \num{2108} triangulations of \textit{hyperbolic} 3-manifolds using
the Regina closed census and Hodgson-Weeks closed hyperbolic census, corresponding to 303 unique census manifolds.
The first such manifold that we saw was generated by the filling bouquet with
edge weights \texttt{(2,1,0,3,2,3,1,1,1)} (total weight 14) and no pre-resolved edges;
we can recreate the filling bouquet on the boundary triangulation of \texttt{eHuGabdes} using these edge weights, as shown in Figure~\ref{fig:hypexample}.
The isomorphism signature of the resulting petal-filled triangulation was
\texttt{kLvAAPPkcefehijijjuxdxasalalw}.
After simplifying this triangulation, we identified the manifold as \texttt{Hyp\_0.94270736} from the Regina closed census.

\begin{figure}[htbp]
    \centering
    \includegraphics[width=0.9\textwidth]{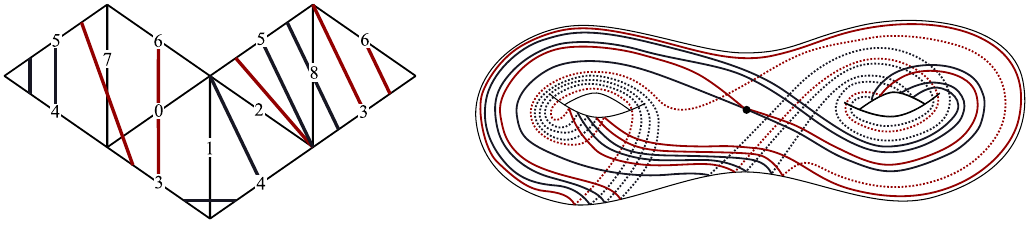}
    \caption{Left: Curves reconstructed from the edge weights $(2,1,0,3,2,3,1,1,1)$.
    Right: How these curves appear in our $\mathbb{R}^3$ embedding of \texttt{eHuGabdes} (recall Figure~\ref{fig:eHuG-for-examples}).}
    \label{fig:hypexample}
\end{figure}

\subsection{Genus 3}
For the genus 3 experiments we  started with the triangulation of a genus-$3$ handlebody given by the isomorphism signature \texttt{hHbLbqiabegeti};
this is a layered triangulation, and its boundary contains 15 edges.

By generating filling bouquets with total edge weight between 3 and 20, we constructed a total of \num{25343} petal-filled triangulations of prime 3-manifolds.
After simplifying these petal-filled triangulations and removing duplicates, we were left with \num{6929} different simplified triangulations.
We successfully identified \num{5428} of these as distinct minimal triangulations in
the Regina closed orientable census and the Hodgson-Weeks census, corresponding to \num{2222} unique census manifolds.

From these \num{2222} census manifolds, we observed 679 manifolds not seen in our genus 2 experiments.
Of these, 83 were identified as lens spaces.
Although we cannot rule out the existence of a genus 2 splitting for the remaining 596 manifolds, we can conclude that they have Heegaard genus either 2 or 3.
In terms of hyperbolic 3-manifolds, we identified \num{1093} triangulations of such manifolds using
the Regina closed census and Hodgson-Weeks closed hyperbolic census, corresponding to 151 unique census manifolds.

\subsection{Cutwidth and treewidth observations}
We make some brief observations about the cutwidth (recall Definition~\ref{defn:cutwidth}) and treewidth.

\begin{definition}[Treewidth]
    Let $G=(V,E)$ be a multigraph. A \textit{tree decomposition} of $G$ is a pair $\big(\{B_i,\; i\in I\},\; (I, F)\big)$
    consisting of \textit{bags} $B_i\subseteq V$, $i\in I$, and a tree $T=(I,F)$ such that the following hold:
    \begin{enumerate}[label=(\roman*)]
        \item Each vertex $v\in V$ is contained in at least one bag.
        That is, $\cup_i B_i = V$.
        \item For each edge $\{v_1,v_2\}\in E$, there is some $i\in I$ such that $\{v_1,v_2\}\subseteq B_i$.
        \item For each $v\in V$, the bags containing $v$ span a connected subtree of $T$.
        That is, ${\{i\in I \colon v\in B_i\}}$ spans a connected subtree of $T$.
    \end{enumerate}
    The \textit{width} of a tree decomposition is $\max_{i\in I}|B_i|-1$. The \textit{treewidth} of a graph is the minimum width over all tree decompositions. The treewidth of a triangulation is the treewidth of its dual graph.
\end{definition}

We used Sage~\cite{sagemath} to compute the cutwidth and treewidth of the triangulations that we found in the experiments discussed above.
The caveat is that the dual graphs of our triangulations are, in general, multigraphs.
Whilst this poses no problems for computing treewidth, it appears that the cutwidth algorithms present in Sage do not count parallel edges;
to the best of the authors' knowledge, there are no algorithms readily available that remedy this.
Thus, unless we have a triangulation whose dual graph happens to be a simple graph,
the cutwidth values computed by Sage are only lower bounds rather than exact values.

We partly circumvented this issue by restricting our attention to the manifolds for which we constructed
petal-filled triangulations with cutwidth bounded below by $4g-2$, as calculated by Sage;
in these cases, we know that this is the exact value of the cutwidth, since it coincides with the upper bound from Theorem~\ref{thm:bounded-cutwidth}.
For these petal-filled triangulations, we then used Sage to obtain a lower bound on the cutwidth of the corresponding simplified triangulation;
for the cases where this lower bound was greater than the cutwidth of the original petal-filled triangulation,
we had verification that the cutwidth had increased after simplification.

From the \num{38855} different simplified triangulations constructed from the genus-$2$ handlebody,
we found that \num{34419} of the corresponding petal-filled triangulations had cutwidth equal to $6$.
We were able to verify that 874 of these petal-filled triangulations had cutwidth increasing after simplification, with the largest verified increase being 2.
More specifically, 254 of these 874 triangulations were verified to have cutwidth increasing by (at least) 2 after simplification;
59 were identified as closed hyperbolic manifolds appearing in the censuses, whilst the remaining 195 could not be identified via census lookup.
The first of these 254 triangulations was identified as \texttt{Hyp\_1.01494161:\#8} in the Regina closed orientable census.
This was generated from the weights \texttt{(2,0,2,4,1,3,3,1,1)}, with no pre-resolved edges.
For this particular example, the dual graph of the simplified triangulation happened to be a simple graph, so the cutwidth is actually equal to 8.

From the \num{6929} different simplified triangulations constructed from the genus-3 handlebody,
Sage calculated all cutwidths of the corresponding petal-filled triangulations to be less than $10$.
As explained above, these calculated values are not guaranteed to be exact, so we did not study the cutwidth of these examples in greater detail.

We now consider treewidth.
For triangulations generated from the genus-$2$ handlebody, we observed \num{1455} instances of
the treewidth increasing after simplifying the petal-filled triangulation.
In particular, we found 5 triangulations where the treewidth increased by 2 after simplification,
though we were unable to find the simplified triangulations in the censuses.
Of these, 4 triangulations were found to have the cutwidth increase by at least 2 after simplification, with the other having cutwidth increase by at least 1.

For triangulations generated from the genus-$3$ handlebody, we observed 0 instances of treewidth increasing after simplifying the petal-filled triangulation.

\appendix
\section{Appendix. A genus-\texorpdfstring{$2$}{2} layered handlebody: \texttt{eHuGabdes}}\label{app:basic-triangulations}

Here we study the 4-tetrahedron triangulation of a genus-$2$ layered handlebody described by the isomorphism signature \texttt{eHuGabdes}. The gluing information is given in Figure~\ref{fig:Regina-eHuGabdes}.

\begin{figure}[h]
\centering\small
  \begin{tabular}{c|c|c|c|c}
    Tetrahedron & Face 012 & Face 013 & Face 023 & Face 123 \\
    \hline
    0 &   --   & 1(023) &   --   & 1(132) \\
    1 & 2(013) & 3(012) & 0(013) & 0(132) \\
    2 & 3(301) & 1(012) &   --   &   --   \\
    3 & 1(013) & 2(120) &   --   &   --   \\
  \end{tabular}
  \caption{Triangulation of a minimal (4 tetrahedron) genus-$2$ handlebody built in Regina from the isomorphism signature \texttt{eHuGabdes}. Note that the labelling has been `oriented' so that the order of vertex labels is consistent between tetrahedra.}
 \label{fig:Regina-eHuGabdes}
\end{figure}

The unglued faces that form the boundary are: 0(012), 0(023), 2(023), 2(123), 3(023) and 3(123). Figure~\ref{fig:eHuG-bdry} shows tetrahedra 0, 2 and 3, viewed from these six faces. Note that the edges are colour-coded/decorated according to edge classes, which can be easily read off from the `Skeletal Components' tab in the Regina GUI.

\begin{figure}[h]
    \centering
\includegraphics[width=0.9\textwidth]{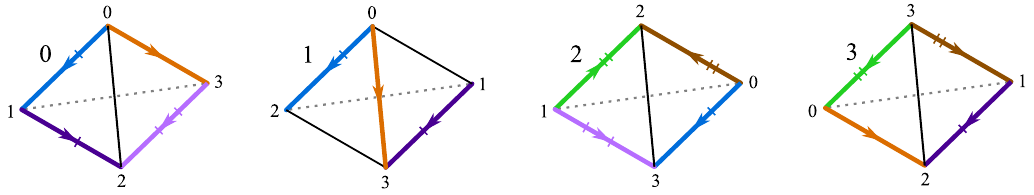}
    \caption{The four tetrahedra with labels corresponding to Figure~\ref{fig:Regina-eHuGabdes}. The coloured edges are those that help us reconstruct the triangulation in Section~\ref{sec:eHuG-layers}.}
    \label{fig:eHuG-bdry}
\end{figure}

\subsection{Tracing layers from the outside inwards}\label{sec:eHuG-layers}
Here we visualise the layering in the \texttt{eHuGabdes} triangulation in such a way that the folding of the innermost layer can be demonstrated (in Section~\ref{sec:eHuG-curves}).

\small
\begin{center}
    \begin{tabular}{m{0.03\textwidth} m{0.5\textwidth} m{0.37\textwidth}}
         (a) & \includegraphics[width=0.4\textwidth]{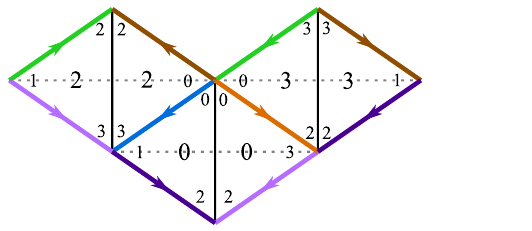} & Identifying the orange edges 0(03) and 3(02) and the blue edges 0(01) and 2(03), we can visualise the boundary triangulation as an octagon.
    \end{tabular}
    \begin{tabular}{m{0.03\textwidth} m{0.5\textwidth} m{0.37\textwidth}}
         (b) & \includegraphics[width=0.4\textwidth]{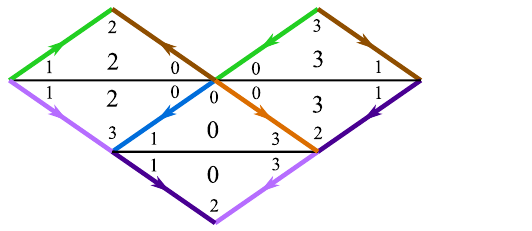} & Considering the back faces of these three tetrahedra, we visualise the next layer in.
    \end{tabular}
    \begin{tabular}{m{0.03\textwidth} m{0.5\textwidth} m{0.37\textwidth}}
         (c) & \includegraphics[width=0.4\textwidth]{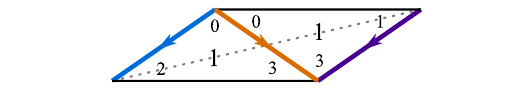} & We may then glue tetrahedron 1 from below, to the faces 0(013) and 3(012).
    \end{tabular}
    \begin{tabular}{m{0.03\textwidth} m{0.5\textwidth} m{0.37\textwidth}}
        (d) & \includegraphics[width=0.4\textwidth]{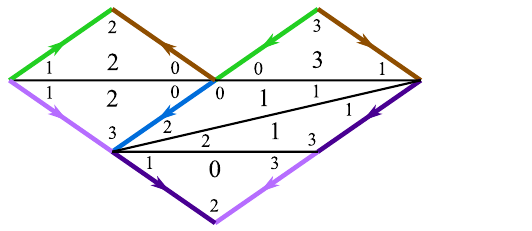} & The innermost layer then consists of the two back faces of tetrahedron 1, the two back faces of tetrahedron 2, and faces 0(123) and 3(013).
    \end{tabular}
    \begin{tabular}{m{0.03\textwidth} m{0.5\textwidth} m{0.37\textwidth}}
         (e) & \includegraphics[width=0.4\textwidth]{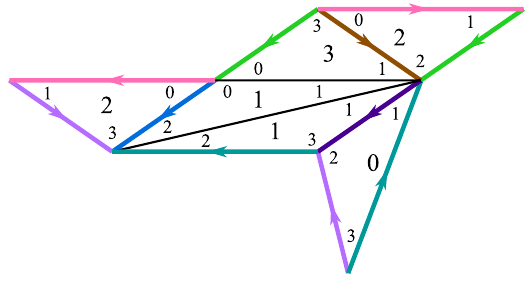} & Here we choose to view the boundary triangulation differently, by shifting which edges make up the boundary of the octagon. More specifically, cut along edges 2(01) and 0(13) and reglue along edges 2(02) and 0(12), respectively.
    \end{tabular}
    \begin{tabular}{m{0.03\textwidth} m{0.5\textwidth} m{0.37\textwidth}}
         (f) & \includegraphics[width=0.4\textwidth]{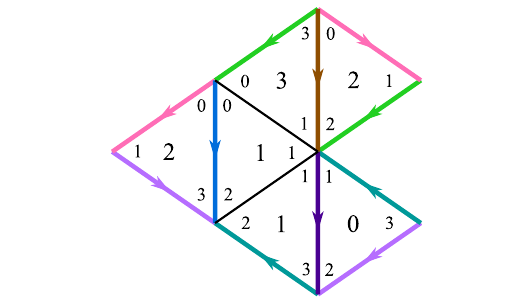} & Finally, we just isotope the previous triangulation into a cleaner position. Note that the adjusted view of the innermost layer was chosen because it makes it clear how the final three face pairings are obtained via folds (see Section~\ref{sec:eHuG-curves}).
    \end{tabular}
\end{center}

\subsection{Identifying meridian curves from folds}\label{sec:eHuG-curves}
Here we use the folds to identify the curves corresponding to meridians of the handles in the \texttt{eHuGabdes} triangulation, and trace back through the layering to visualise them on the boundary.

\small
\begin{center}
    \begin{tabular}{m{0.03\textwidth} m{0.5\textwidth} m{0.37\textwidth}}
         (a) & \includegraphics[width=0.4\textwidth]{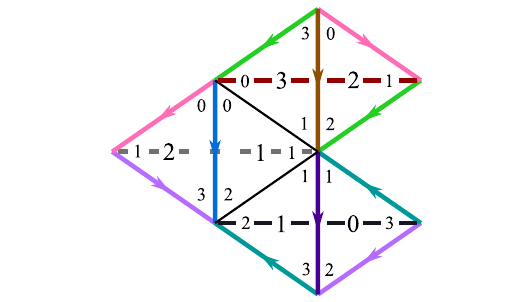} & The three folds are made across the three vertical edges. That is, 2(012) is identified to 3(301), 1(012) is identified to 2(013), and 0(123) is identified to 1(132). The curves that are made homotopically trivial by these folds are indicated by the dashed lines.
    \end{tabular}
    \begin{tabular}{m{0.03\textwidth} m{0.5\textwidth} m{0.37\textwidth}}
         (b) & \includegraphics[width=0.4\textwidth]{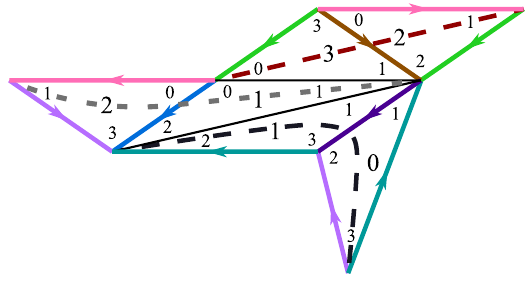} & We then retrace the steps in Section~\ref{sec:eHuG-layers}, keeping track of these curves. Observe that the middle dashed line separates the two handles of the genus-$2$ boundary, meaning this is not one of the meridians. We therefore ignore it from here on.
    \end{tabular}
    \begin{tabular}{m{0.03\textwidth} m{0.5\textwidth} m{0.37\textwidth}}
         (c) & \includegraphics[width=0.4\textwidth]{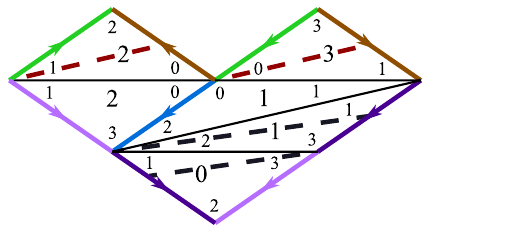} & How the red and blue filling petals appear with respect to part (d) in Section~\ref{sec:eHuG-layers}.
    \end{tabular}
    \begin{tabular}{m{0.03\textwidth} m{0.5\textwidth} m{0.37\textwidth}}
         (d) & \includegraphics[width=0.4\textwidth]{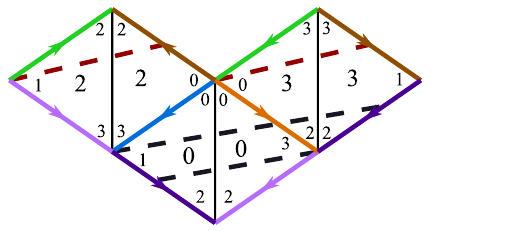} & How the red and blue filling petals appear with respect to the boundary triangulation.
    \end{tabular}
\end{center}

\begin{figure}[h!]
    \centering
    \includegraphics[width=0.8\textwidth]{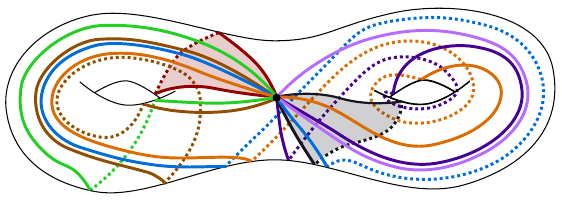}
    \caption{An embedding of the \texttt{eHuGabdes} triangulation into $\mathbb{R}^3$. To avoid unnecessary complication we have omitted the uncoloured edges of the final triangulation (d) in Section~\ref{sec:eHuG-curves}. The red and blue shading indicates the locations of the compression discs corresponding to each of the two handles.}
    \label{fig:eHuG-embedding}
\end{figure}

\bibliography{handlebodies}

\end{document}